\documentclass[11pt]{amsart}
\usepackage{verbatim, latexsym, amssymb, amsmath,color,enumitem}
\usepackage{hyperref} 
\usepackage{epsfig}
\usepackage{bm}
\usepackage{mathrsfs}
\usepackage[margin=1.25in]{geometry}
\usepackage{subcaption}
\usepackage{soul}

\newtheorem{theorem}{Theorem}[section]

\newtheorem{lemma}[theorem]{Lemma}

\newtheorem{proposition}[theorem]{Proposition}

\theoremstyle{definition}
\newtheorem{definition}[theorem]{Definition}

\newtheorem{remark}[theorem]{Remark}

\newtheorem*{acknowledgements}{Acknowledgements}

\numberwithin{equation}{section}

\newcommand{\V}{\mathcal{V}}

\newcommand{\R}{\mathbb{R}}
\newcommand{\N}{\mathbb{N}}
\newcommand{\mH}{\mathcal{H}}

\newcommand{\Pcal}{\mathcal{P}}

\newcommand{\lan}{\langle}
\newcommand{\ran}{\rangle}

\newcommand{\Om}{\Omega}

\newcommand{\ka}{\kappa}

\newcommand{\mB}{\mathbb{B}}
\newcommand{\mZ}{\mathbb{Z}}

\newcommand{\mS}{\mathbb{S}}

\newcommand{\tM}{\widetilde{M}}

\newcommand{\spt}{\operatorname{spt}}
\newcommand{\dist}{\operatorname{dist}}
\newcommand{\Div}{\operatorname{div}}

\newcommand{\Hess}{\operatorname{Hess}}

\newcommand{\interior}{\operatorname{int}}
\newcommand{\Ric}{\operatorname{Ric}}
\newcommand{\Clos}{\operatorname{Clos}}

\newcommand{\rom}[1]{\expandafter\romannumeral #1}

\newcommand{\bd}{\partial}

\newcommand{\clos}{\operatorname{Clos}}
\newcommand{\laa}{\langle}
\newcommand{\raa}{\rangle}

\newcommand{\grad}{\nabla}
\newcommand{\del}{\nabla}
\newcommand{\ric}{\operatorname{Ric}}
\renewcommand{\div}{\operatorname{div}}
\newcommand{\eps}{\varepsilon}

\title[Curvature estimates for stable FBMHs in locally wedge-shaped manifolds]{Curvature estimates for stable free boundary minimal hypersurfaces in locally wedge-shaped manifolds}

\author[Liam Mazurowski]{Liam Mazurowski}
\address{Cornell University, Department of Mathematics, Ithaca, New York 14850}
\email{lmm334@cornell.edu}

\author{Tongrui Wang}
\address{Institute for Theoretical Sciences, Westlake Institute for Advanced Study, Westlake University, Hangzhou, Zhejiang, 310024, China}
\email{wangtongrui@westlake.edu.cn}

\begin{document}

\begin{abstract}
	In this paper, we consider locally wedge-shaped manifolds, which are Riemannian manifolds that are allowed to have both boundary and certain types of edges. 
    We define and study the properties of free boundary minimal hypersurfaces inside locally wedge-shaped manifolds. 
    In particular, we show a compactness theorem for free boundary minimal hypersurfaces with curvature and area bounds in a locally wedge-shaped manifold. 
    Additionally, using Schoen-Simon-Yau’s estimates, we also prove a Bernstein-type theorem indicating that, under certain conditions, a stable free boundary minimal hypersurface inside a Euclidean wedge must be a portion of a hyperplane. 
    As our main application, we establish a curvature estimate for sufficiently regular free boundary minimal hypersurfaces in a locally wedge-shaped manifold with certain wedge angle assumptions. We expect this curvature estimate will be useful for establishing a min-max theory for the area functional in wedge-shaped spaces.
\end{abstract}


\maketitle

\section{Introduction}

Let $M$ be a manifold with boundary. A free boundary minimal hypersurface $\Sigma$ in $M$ is a critical point of the area functional whose boundary $\bd \Sigma$ is constrained to lie within $\bd M$ but can otherwise move freely.  Courant \cite{courant2005dirichlet} was among the first to investigate the properties of free boundary minimal surfaces.  A typical problem he considered is the following: given a smooth torus $S$ embedded in $\R^3$, find a disk type minimal surface with free boundary on $S$ which ``spans the hole'' in $S$. Using the mapping method developed by Douglas and Rado to solve the Plateau problem, Courant proved the existence of a solution in the form of a conformal, Dirichlet energy minimizing map $D^2\to \R^3$. Later Str\"uwe \cite{struwe1984free} showed the existence of an unstable disk type solution to the same problem. Using techniques from geometric measure theory, Gr\"uter and Jost \cite{gruter1986allard} proved the existence of an unstable, embedded minimal disk with free boundary inside any convex domain in $\R^3$ with smooth boundary.

In the 1960s, Almgren began to develop a variational theory for the area functional on Riemannian manifolds based on the tools of geometric measure theory.  While, in general, a manifold $M$ may not contain any area minimizing (free boundary) minimal hypersurfaces, Almgren \cite{almgren1965theory} was able to use min-max methods to prove the existence of stationary varifolds in any $M$.  These stationary varifolds are a weak measure theoretic notion of minimal hypersurfaces.  (Almgren's theory in fact produces stationary varifolds of arbitrary codimension, but in this paper we will only consider the codimension one case.) In the early 1980s, Pitts \cite{pitts2014existence} showed that when $M$ is closed and $3 \le \dim(M)\le 6$, the stationary varifold found by Almgren is actually a smooth, closed, embedded minimal hypersurface. A crucial ingredient for establishing the regularity is a curvature estimate for stable minimal hypersurfaces first proved by Schoen-Simon-Yau \cite{schoen1975curvature}.  This curvature estimate was later improved by Schoen and Simon \cite{schoen1981regularity}. Their improved curvature estimate implies that Almgren's stationary varifold is in fact smooth off of a set of codimension 7.  

When $M$ has a non-empty boundary, the variational theory of Almgren still gives the existence of a stationary varifold with free boundary. Li and Zhou \cite{li2021min} established the regularity theory for min-max minimal hypersurfaces in this free boundary setting. 
Again curvature estimates for stable free boundary minimal hypersurfaces play an important role in establishing regularity. In the free boundary setting, these curvature estimates can be established by blowing up to get a stable free boundary hypersurface in a half-space, and then reflecting to get a stable minimal hypersurface (with no boundary) in Euclidean space. 

In a series of four papers in the 1990s, Hildebrandt and Sauvigny (\cite{hildebrandt1997minimal1}\cite{hildebrandt1997minimal2}\cite{hildebrandt1999minimal3}\cite{hildebrandt1999minimal4}) studied free boundary minimal surfaces inside a wedge in $\R^3$. Among other things, they found that the behavior of such minimal surfaces depends crucially on the wedge angle $\theta$.  Recently, there has been renewed interest in understanding free boundary minimal hypersurfaces and related classes of surfaces in less regular domains. For instance, Edelen and Li \cite{edelen2022regularity} have proved an {Allard-type} theorem for free boundary minimal hypersurfaces in locally polyhedral domains. Their result in particular implies the partial regularity of area minimizing free boundary minimal hypersurfaces in locally polyhedral domains.  Also, capillary surfaces in polyhedrons were the key ingredient in Li's \cite{li2020polyhedron} proof of Gromov's dihedral rigidity conjecture. {Moreover, the equivariant min-max theory established by the second author in \cite{wang2022min}\cite{wang2023min} can also be seen as a variational theory in the singular orbit space.} 

Ultimately, it would be very interesting to establish a free boundary min-max theory in less regular spaces. To accomplish this, a crucial first step is to prove curvature estimates for stable free boundary minimal hypersurfaces in these spaces. The goal of this paper is to prove such curvature estimates for free boundary minimal hypersurfaces in a class of spaces we call {\em locally wedge-shaped manifolds}. We also prove a {Bernstein-type} theorem and a compactness theorem for stable free boundary minimal hypersurfaces in locally wedge-shaped manifolds.  Before stating the theorems, we first need to explain the terminology. 

Roughly speaking, a locally wedge-shaped manifold $M^{n+1}$ consists of 3 strata: an $n+1$ dimensional set of interior points where $M$ looks locally like $\R^{n+1}$, an $n$ dimensional set of face points where $M$ is locally modeled on a half-space $\R^{n+1}_+$, and an $n-1$ dimensional set of edge points where $M$ is locally modeled on a wedge in $\R^{n+1}$. 
The meeting angle $\theta$ between the two faces of $M$ is allowed to vary along an edge, although we assume that $0< \theta < \pi$ at all edge points. We write $\operatorname{int}(M)$ for the set of interior points of $M$, $\bd^F M$ for the set of face points of $M$, and $\bd^E M$ for the set of edge points of $M$. See Section \ref{Sec: preliminary} for precise definitions. 
\begin{figure}[h]
\centering
\includegraphics[width=2in]{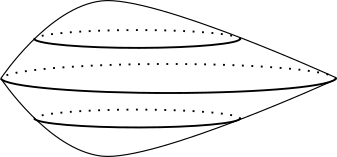}
\caption{A locally wedge-shaped manifold}
\label{fig1}
\end{figure}
Figure \ref{fig1} illustrates a simple locally wedge-shaped manifold $M^3$ which looks like a mussel. There is a single edge along the center, two faces formed by the upper and lower caps, and an interior region bounded by the two faces. Note that the angle between the two faces is allowed to vary along the edge.

Let $\Sigma^n\subset M^{n+1}$ be a locally wedge-shaped manifold of dimension one less than $M$.  We call such a $\Sigma$ a locally wedge-shaped hypersurface ($1$-codimentional submanifold) of $M$. We say that $\Sigma$ is {\em properly embedded} in $M$ if $\bd^E \Sigma \subset \bd^E M$ and $\bd^F \Sigma \subset \bd^F M$ and $\operatorname{int}(\Sigma)\subset \operatorname{int}(M)$.  As in the case of manifolds with {smooth} boundary, it may happen that the limit of a sequence of properly embedded locally wedge-shaped submanifolds fails to be properly embedded. Thus we shall also consider {\em almost properly embedded} locally wedge-shaped submanifolds. Here we say $\Sigma\subset M$ is almost properly embedded if $\bd^E\Sigma\subset \bd^E M$ and $\bd^F\Sigma \subset \bd^E M \cup \bd^F M$ and $\operatorname{int}(\Sigma)\subset \bd^F M \cup \operatorname{int}(M)$. 
Thus, when $\Sigma$ is almost properly embedded, we allow interior points of $\Sigma$ to contact the faces of $M$ and we allow face points of $\Sigma$ to contact the edge of $M$. 

In Section \ref{Sec: variations in wedge manifold}, we define free boundary minimal hypersurfaces in a locally wedge-shaped manifold $M$.  The definition requires some care since we would like to include classical properly embedded free boundary hypersurfaces without edges, while also obtaining a closed condition. Assume $\Sigma^n\subset M^{n+1}$ is an almost properly {embedded} locally wedge-shaped hypersurface. We define a class of vector fields $\mathfrak X(M,\Sigma)$ depending on both $M$ and $\Sigma$. This class has the property that for any $X$ in $\mathfrak X(M,\Sigma)$, the flow of $X$ preserves the almost proper embeddedness of $\Sigma$ for a short time. Moreover, vector fields in $\mathfrak X(M,\Sigma)$ are allowed to push interior points of $\Sigma$ off the faces of $M$ and to push face points of $\Sigma$ off the edge of $M$. We then say $\Sigma$ is a free boundary minimal hypersurface if the first variation of area vanishes for all vector fields in $\mathfrak X(M,\Sigma)$. This turns out to be equivalent to asking that the mean curvature of $\Sigma$ vanishes at all points $p\in \operatorname{int}(\Sigma)$ and that the unit co-normal $\eta$ to $\Sigma$ is perpendicular to at least one face of $M$ at all points in $\bd^F\Sigma$. 
\begin{figure}[h]
\centering
\begin{subfigure}{0.4\linewidth}
\centering
\includegraphics[width=2in]{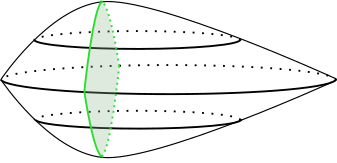} 
\caption{A free boundary minimal surface}
\end{subfigure}
\hspace{1cm}
\begin{subfigure}{0.3\linewidth}
\centering
\includegraphics[width=2in]{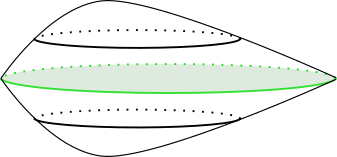}
\caption{A non-example}
\end{subfigure}
\caption{}
\end{figure}
In Figure 2(A), $\Sigma$ is properly embedded with two edge points. This surface $\Sigma$ is a free boundary minimal hypersurface since the mean curvature vanishes in the interior and the outward unit co-normal $\eta(p)$ is perpendicular to at least one  face of $M$ at all $p\in \bd^F\Sigma$. In Figure 2(B), $\Sigma$ is almost properly embedded but not properly embedded. This surface $\Sigma$ is not a free boundary minimal surface according to our definition since the outward unit co-normal along the face of $\Sigma$ is not perpendicular to either face of $M$.

The regularity of free boundary minimal hypersurfaces in a locally wedge-shaped domain seems to be a delicate issue. By analogy with the Neumann problem for the laplacian in a wedge, we expect the optimal regularity of a free boundary minimal hypersurface $\Sigma^n \subset M^{n+1}$ is only $C^{1,\alpha}$ along the edge of $\Sigma$, even in the area minimizing case.  Because this paper deals with curvature estimates, we will always assume a priori that our free boundary minimal hypersurfaces are smooth away from the edge and at least $C^{1,1}$ up to and including the edge. 
In particular, if on each connected component $\bd^E_i M$ of $\bd^E M$, either
\begin{align}
    &\theta(p)<\frac{\pi}{2} \mbox{ for all } p\in \bd^E_i M ; \mbox{ or } \tag{$\dagger$}\label{dag}
    \\ &\theta(p)\equiv\frac{\pi}{2} \mbox{ and } A_{\bd^+M}(\eta_+, v) = A_{\bd^-M}(\eta_-, v) \mbox{ for all } p\in \bd^E_i M, v\in T_p(\bd^EM), \tag{$\ddagger$}\label{ddag}
\end{align}
where $A_{\bd^{\pm}M}$ is the second fundamental form (with respect to the inward unit normal) of the two faces $\bd^{\pm}M$ of $M$ meeting along $\bd^E_i M$, and $\eta_{\pm}$ is the outward unit co-normal of $\bd^{\pm}M$ along $\bd^E M$. 
Then we show in another paper (\cite[Theorem 2.17]{mazurowski2023min}) that the regularity of free boundary minimal hypersurfaces can be upgraded to $C^{2,\alpha}$ up to and including the edge so that this priori $C^{1,1}$ regularity assumption is automatically satisfied.
Indeed, the elliptic estimates (\cite[Appendix C]{mazurowski2023min}) are also needed to obtain the curvature estimates in the final blow-up argument.

We can now state our main results. {For simplicity, we often abbreviate `free boundary minimal hypersurfaces' as `FBMHs'.} The first is a compactness theorem for locally wedge-shaped free boundary minimal hypersurfaces with bounded curvature and area. 

\begin{theorem}
    \label{theorem:main1}
    Let $M^{n+1}$ be a locally wedge-shaped {Riemannian} manifold. Let $\{\Sigma_i^n\}_{i\in\N}\subset M^{n+1}$ be a sequence of $C^{1,1}$-to-edge almost properly embedded, locally wedge-shaped, FBMHs so that 
    $$ \sup_{i}{\rm Area}(\Sigma_i)\leq C \qquad {\rm and} \qquad \sup_{i}\sup_{\Sigma_i}|A_{\Sigma_i}|\leq C $$
    for some constant $C>0$. 
    Then after passing to a subsequence, $\Sigma_i$ converges $C^{1,\alpha}$-to-edge $(\forall \alpha\in (0,1))$, $C^\infty$ away from the edge, locally uniformly to an almost properly embedded $C^{1,1}$-to-edge locally wedge-shaped hypersurface $\Sigma\subset M$ {(with multiplicity)} which is also minimal in $M$ with free boundary. 
\end{theorem}

The second main result is a {Bernstein-type} theorem which says that stable locally wedge-shaped FBMHs inside a Euclidean wedge with Euclidean volume growth must be planes. See Section \ref{Sec: variations in wedge manifold} for a discussion of stability in our setting. Here we obtain a stronger result when the wedge angle $\theta$ is no more than $\pi/2$. Because global reflection arguments are no longer available, our strategy is to directly follow Schoen-Simon-Yau's original proof in a wedge. We then use local reflection arguments to show that the extra boundary terms appearing in the formulas actually vanish. 

\begin{theorem}
\label{theorem:main2}
Fix a dimension $3\le n+1\le 6$ and let $\Omega^{n+1}$ be an $(n+1)$-dimensional wedge domain in $\R^{n+1}$ with wedge angle $\theta \le \pi/2$. Assume that $\Sigma$ is a $2$-sided $C^{1,1}$-to-edge almost properly embedded FBMH in $\Omega$ which is stable according to Definition \ref{definition:stability}. Further assume that $\Sigma$ has Euclidean volume growth: there is a constant $C > 0$ such that 
    \[
    \mathcal H^{n}(\Sigma\cap \mathbb {B}^{n+1}_r(0))\le Cr^n
    \]
    for all $r > 0$. Then $\Sigma = \Omega \cap P$ where $P\subset \R^{n+1}$ is an $n$-plane.  If $\theta > \pi/2$, the conclusion still holds provided we assume that $\Sigma$ is properly embedded. 
\end{theorem}

Finally, combining the above two theorems with a blowup argument, we conclude a curvature estimate for locally wedge-shaped stable FBMHs as a main application. 
As in \cite{pitts2014existence}, this curvature estimate is an essential ingredient in the min-max construction for FBMHs within locally wedge-shaped Riemannian manifolds developed later by the two authors in \cite{mazurowski2023min}.

\begin{theorem}
    \label{theorem:main3}
    Let $3\leq n+1\leq 6$, and let $M^{n+1}\subset \R^L$ be a locally wedge-shaped Riemannian manifold (with induced metric) satisfying (\ref{dag}) or (\ref{ddag}). There is an $r_0>0$ such that the following holds. 
    Suppose $p\in\bd^EM$, $r\in (0,r_0)$, and  $\Sigma \subset M \cap \mB^L_r(p)$ is an almost properly embedded, $C^{1,1}$-to-edge locally wedge-shaped, stable FBMH so that ${\rm Area}(\Sigma\cap \mB^L_r(p)) \leq C_1$. 
    Then 
    $$ \sup_{x\in \Sigma\cap\mB_{r/2}^L(p)} |A_\Sigma|(x) \leq C_2, $$
    where $C_2=C_2(C_1,M)>0$. 
 \end{theorem}

\begin{remark} 
    By \cite[Theorem 2.17]{mazurowski2023min} and (\ref{dag})(\ref{ddag}), the a priori $C^{1,1}$ regularity can be weaken to $C^{1,\alpha}$. 
    We also note that, in the area minimizing case, Li \cite{li2019dihedral} proved curvature estimates in more general Riemannian polyhedra. Our result can be seen as answering, in a certain sense, \cite[Remark 3.9]{li2019dihedral} about curvature estimates in the stable case, at least for locally wedge-shaped domains. The key point is that in an acute wedge domain $\Omega = W\times \R$ where $W = \{(r,\theta): r \ge 0,\, \theta \in (-\theta_0/2, \theta_0/2)\}$, the potentially troublesome vertical plane $\Sigma = \{\theta = \theta_1\}$ does not qualify as stationary according to our definitions.
\end{remark}

\subsection{Organization} The remainder of the paper is organized as follows. In Section \ref{Sec: preliminary}, we formally define locally wedge-shaped manifolds and locally wedge-shaped submanifolds. In Section \ref{Sec: variations in wedge manifold}, we discuss the first and second variation of area for locally wedge-shaped submanifolds and define locally wedge-shaped free boundary minimal hypersurfaces. In Section \ref{sec:compactness}, we prove Theorem \ref{theorem:main1} on the compactness of locally wedge-shaped FBMHs with controlled area and curvature. In Section \ref{sec:Bernstein} we prove the {Bernstein-type} result (Theorem \ref{theorem:main2}). Finally in Section \ref{Sec: curvature estimates}, we prove Theorem \ref{theorem:main3} regarding the curvature estimates for stable locally wedge-shaped FBMHs. 
    
\begin{acknowledgements}
	The authors would like to thank Professor Xin Zhou for suggesting the problem and for many helpful discussions. The second author also would like to thank Professor Gang Tian for his constant encouragement, and thank Professor Xin Zhou and Cornell University for their hospitality. 
	The second author is partially supported by China Postdoctoral Science Foundation 2022M722844.
\end{acknowledgements}
	

\section{Preliminary}\label{Sec: preliminary}

In this section, we introduce the definition of {\em locally wedge-shaped manifolds} and {\em almost properly embedded locally wedge-shaped hypersurfaces}.  
To begin with, let us fix some notation in the Euclidean space, which is partially borrowed from \cite{edelen2022regularity}. 

\medskip
\subsection{Wedge domains}
Let $H^m_{+}$ and $H^m_{-}$ be two closed half spaces in $\R^{m}$, where $m\in\{2,3,\dots\}$. 
Then we say $$\Omega^m = H^m_{+}\cap H^m_-$$ is an $m$-dimensional {\em wedge domain} if it has non-empty interior. 
By rotating, we can always write $\Omega^m$ in the form of 
\begin{align}\label{Eq: standard wedge}
	\Omega^m &= \Omega_\theta^2 \times \R^{m-2} 
	\\ &= \Clos\big(\big\{(x_1,\dots,x_m)\in \R^m: x_1> 0, x_2\in \left(\tan\left(-\theta/2\right) x_1, ~\tan\left(\theta/2\right) x_1\right) \big\} \big) , \nonumber
\end{align}
where $\theta\in (0,\pi]$ is called the {\em wedge angle} of $\Omega^m$. 
We also use the notation $\Omega^m_\theta$ to emphasize its angle. 
In particular, if $\theta = \pi$, then $\Omega^m_\pi = \{(x_1,\dots,x_m)\in \R^m: x_1\geq 0\}$ is a half space, and we call $\Omega^m_\pi$ a {\em trivial} wedge domain. 

For a non-trivial wedge domain $\Omega^m_\theta$, i.e. $0<\theta<\pi$, we define the following {\em vertical} half hyperplanes in $\Om^m_\theta$: 
\begin{eqnarray}\label{Eq: vertical hyperplane in wedge}
	P_\alpha^{m-1} &:=& \left\{(x_1,x_2,\dots,x_m)\in \R^m: x_1\geq 0, x_2=\tan\left(\alpha\right)x_1 \right\},
	\\ \partial^{\pm} \Omega^m_\theta &=& P^{m-1}_{\pm\frac{\theta}{2}},
\end{eqnarray}
where $ \alpha\in [-\theta/2, \theta/2] $. 
Without loss of generality, we assume $\partial^{\pm}\Omega^m \subset \partial H^m_{\pm}  $. 
Moreover, we define the {\em horizontal} hyperplanes of $\Omega^m=\Omega^2_\theta\times \R^{n-2}$ by 
\begin{equation}\label{Eq: horizontal plane in domain}
	\mathcal{P}_{\Omega^m} := \{\R^2\times W^{m-3}: \Omega_\theta^2\subset \R^2, W^{m-3} \mbox{ is an $(m-3)$-subspace of }\R^{m-2} \}.
\end{equation}

\begin{definition}[Stratification]\label{Def: stratification of wedge}
	For an $m$-dimensional wedge domain $\Omega = \Omega^m_\theta:=\Omega_\theta^2 \times \R^{m-2} $ with $\theta\in (0,\pi]$, define the stratification $\partial_{m-2}\Omega \subset \partial_{m-1}\Omega \subset \partial_m\Omega$ of $\Omega$ by
	$$ \partial_m\Omega := \Omega, 
	\quad \partial_{m-1}\Omega := \partial \Omega, 
	\quad \partial_{m-2}\Omega:= \left\{ \begin{array}{ll}
		 \{0\}\times \R^{m-2}, ~&\theta \in (0,\pi),
		 \\ \emptyset ,~ &\theta=\pi.
	\end{array} \right.$$
	Then we define
	\begin{itemize}
		\item $\interior(\Omega) := \partial_m\Omega\setminus\partial_{m-1}\Omega$ to be the {\em interior} of $\Omega$;
		\item $\partial^F\Omega := \partial_{m-1}\Omega\setminus\partial_{m-2}\Omega$ to be the {\em face} of $\Omega$;
		\item $\partial^E\Omega := \partial_{m-2}\Omega$ to be the {\em edge} of $\Omega$. 
	\end{itemize}
\end{definition}

Note that the edge $\partial^E\Omega^m_\theta$ is empty if and only if $\Omega^m_\theta$ is a trivial wedge domain, i.e. $\theta=\pi$. 
Additionally, the tangent space $T_x\Omega^m$ of $\Omega^m$ at $x$ is given by 
$$ T_x\Omega^m = \left\{ \begin{array}{ll}
		 	\R^m, ~&x \in \interior(\Omega^m) ,
		 \\ H^m_{\pm} ,~ &x \in \partial^F \Omega^m \cap H^m_{\pm},
		 \\ \Omega^m, ~& x\in \partial^E\Omega^m.
	\end{array} \right. $$

\subsection{Locally wedge-shaped manifolds}

In this subsection, we introduce the notations for manifolds that are locally modeled by wedge domains. 
For simplicity, we denote by $\mB^L_r(p)$ the $L$-dimensional Euclidean open ball, and denote by $0^m$ the origin in $\R^m$. 

\begin{definition}[Locally wedge-shaped manifolds]\label{Def: wedge manifold}
	Let $m,L\in \{2,3,\dots\}$, and $M^m\subset \R^L$ be a (compact) $m$-dimensional manifold with (possibly empty) boundary $\partial M$. 
	Then we say $M$ is a {\em locally wedge-shaped $m$-manifold} if for any $p\in M$, there exist $R=R(p)>0$ and a diffeomorphism $\phi = \phi_p: \mB^L_{R}(0)\to \mB^L_{R}(p)$ so that 
	\begin{itemize}
		\item[(i)] $\phi(0) = p$ and the tangent map $(D\phi)_0\in O(L)$ is an orthogonal transformation;
		\item[(ii)] $\phi\big((\Omega \times 0^{L-m} ) \cap \mB^L_R(0)\big) = M\cap \mB^L_R(p)$, where
			$$ \Omega = \Omega(p) = \left\{ \begin{array}{ll}
		 		\R^m, ~&p\in \interior(M) ,
			 	\\ \mbox{an $m$-dimensional wedge domain } \Omega^m_\theta(p) ,~ &p\in \partial M,
				\end{array} \right. $$
			for some $\theta = \theta(p)\in (0, \pi]$. 
	\end{itemize}
	We call $(\phi, \mB_R^L(p), \Omega)$ a {\em local model} of $M$ around $p$, and call $\theta$ the {\em wedge angle} of $M$ at $p\in \partial M$. 
    {
    Additionally, given $l\in\mZ_+$ and $\alpha\in (0,1]$, if for any $p\in\bd M$ with $0<\theta(p)<\pi$, $\phi_p$ is globally a $C^{l,\alpha}$-diffeomorphism, and is a $C^\infty$-diffeomorphism in $ (\R^{m}\times\{0^{L-m}\})\setminus \bd^E\Omega(p)$, then we say $M$ is a {\em $C^{l,\alpha}$-to-edge} locally wedge-shaped $m$-manifold.}
\end{definition}

\begin{remark}\label{Rem: intrinsic wedge angle}
		In general, we equip $M$ with a Riemannian metric $g_{_M}$ induced from the Euclidean metric $g_0$ in $\R^L$. 
		Then by (i) in the above definition, $\theta=\theta(p)$ is also the {\em intrinsic} wedge angle of $M$ at $p\in\partial M$, which does not depend on the choice of the local model. 
\end{remark}

For a compact ($C^{l,\alpha}$-to-edge) locally wedge-shaped manifold $M^m\subset \R^L$, we can have the following observations from Definition \ref{Def: wedge manifold}. 
Firstly, we see $\partial M$ is smooth at $p$ (in the usual sense) if and only if $\theta(p)=\pi$. 
Then we can claim that 
\begin{equation}\label{Eq: edge is closed manifold}
	\{p\in\partial M: \theta(p)\neq \pi\} \mbox{ is a closed $C^{l,\alpha}$ embedded $(m-2)$-submanifold,}
\end{equation}
and thus $\mH^{m-1}(\{p\in\partial M: \theta(p)\neq \pi\}) = 0$. 
Indeed, take any $p\in\partial M$ with a local model $(\phi, \mB_R^L(p), \Omega)$. 
If $\theta(p)=\pi$, then by definition, $\partial M \cap \mB^L_R(p)$ is  diffeomorphic to an $(m-1)$-plane $\partial \Omega^m_\pi$, which implies the smoothness of $\partial M \cap \mB^L_R(p)$. 
Hence, $\theta(q)=\pi$ for all $q\in \partial M \cap \mB^L_R(p)$, and thus $\{p\in\partial M: \theta(p) = \pi\}$ is an open subset of $\partial M$. 
On the other hand, if $\theta(p)<\pi$, then we have $\theta(q)<\pi$ for $q\in\phi(\bd^E\Omega\cap\mB^L_R(0))$ and $\theta(q)=\pi$ for $q\in\phi(\bd^F\Omega\cap\mB^L_R(0))$ since a non-trivial wedge domain can not be $C^1$-diffeomorphic to a half-space. 
Therefore, $\{p\in\partial M: \theta(p)\neq \pi\}$ is compact and locally $C^{l,\alpha}$ diffeomorphic to $\partial_{m-2}\Omega\cong \R^{m-2}$, which gives the claim.  

Additionally, the local model $(\phi, \mB_R^L(p), \Omega)$ of $M^m$ around $p\in\partial M$ also gives us a local extension of $M$. 
To be exact, define a $C^{l,\alpha}$ embedded $m$-manifold $\tM_{p,R} \subset\mB^L_R(p)$ by
\begin{equation}\label{Eq: local extension}
	\tM_{p,R} := \phi\big((\R^m \times 0^{L-m} ) \cap \mB^L_R(0)\big), 
\end{equation}
where $\R^m\supsetneq \Omega$. 
Then we see $M_{p,R} := M \cap \mB^L_R(p)$ is a wedge-shaped domain in $\tM_{p,R}$. 
Moreover, we define the following notations for simplicity: 
\begin{itemize}
	\item $\partial^{\pm} M_{p,r} := \phi\big((\partial^{\pm}\Omega^m_\theta(p) \times 0^{L-m} ) \cap \mB^L_R(0)\big)$;
	\item $\tM_{p,r}^\pm := \phi\big((H^m_{\pm} \times 0^{L-m} ) \cap \mB^L_R(0)\big)$;
	\item $\partial \tM_{p,r}^\pm := \phi\big((\partial H^m_{\pm} \times 0^{L-m} ) \cap \mB^L_R(0)\big)$.
\end{itemize}
Note $\tM^{\pm}_{p,r}$ is an extension of $M_{p,r}$ across $\bd^{\mp}M_{p,r}$, and $\partial \tM_{p,r}^\pm$ is an extension of $\partial^{\pm} M_{p,r}$. See Figure \ref{fig:local-extensions}.

\begin{figure}[h]
\centering
\begin{subfigure}{0.4\linewidth}
\centering
\includegraphics[height=1.75in]{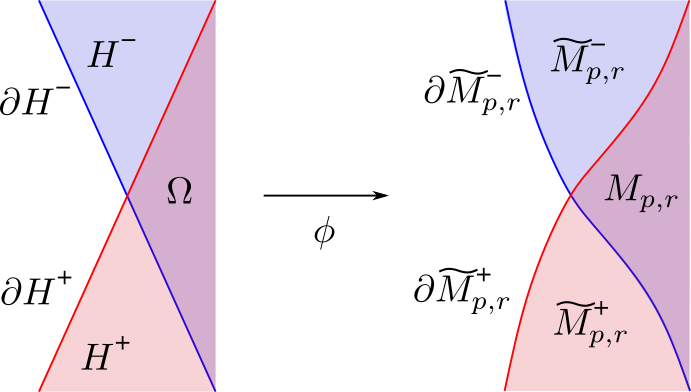} 
\end{subfigure}
\hspace{1cm}
\begin{subfigure}{0.5\linewidth}
\raggedleft
\includegraphics[height=1.75in]{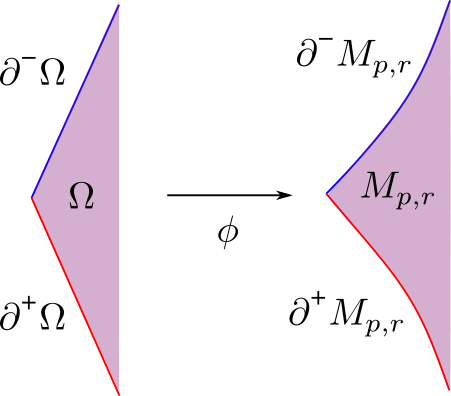}
\end{subfigure}
\caption{Local Extensions of $M$}
\label{fig:local-extensions}
\end{figure}

Similar to Definition \ref{Def: stratification of wedge}, a locally wedge-shaped manifold can be stratified using its local models. 
\begin{definition}\label{Def: stratification of wedge manifold}
	Let $M^m\subset \R^L$ be a compact locally wedge-shaped $m$-manifold. 
	Define the stratification $\partial_{m-2}M\subset \partial_{m-1}M\subset \partial_mM$ of $M$ by
	$$ \partial_m M := M, \quad \partial_{m-1}M := \partial M, \quad \partial_{m-2}M := \cup_{p\in\partial M} \phi_p\big((\partial_{m-2}\Omega^m_{\theta}(p)\times 0^{L-m}) \cap \mB^L_R(0)\big),$$
	where $(\phi_p, \mB^L_R(p), \Omega^m_\theta(p))$ is the local model of $M$ around $p$. 
	Moreover, define
	\begin{itemize}
		\item $\interior(M) := \partial_m M\setminus\partial_{m-1}M$ to be the {\em interior} of $M$;
		\item $\partial^FM := \partial_{m-1}M \setminus\partial_{m-2}M$ to be the {\em face} of $M$;
		\item $\partial^EM := \partial_{m-2}M$ to be the {\em edge} of $M$. 
	\end{itemize}
\end{definition}

Note $\partial_{m-2}M^m = \partial^EM = \{p\in\partial M : \mbox{the wedge angle } \theta(p)\in (0,\pi) \}$ is a closed  $C^{l,\alpha}$ embedded $(m-2)$-submanifold by definitions and (\ref{Eq: edge is closed manifold}). 
Hence, the term `$C^{l,\alpha}$-to-edge' in Definition \ref{Def: wedge manifold} simply means $M$ is $C^{l,\alpha}$ near the edge and $C^\infty$ away from the edge.

Additionally, for $p\in\partial^E M^m$ with local model $(\phi, \mB^L_R(p), \Omega)$, we have $T_pM = (D\phi)_0(T_0\Omega\times 0^{L-m})$. 
Then we further define the {\em vertical} half-hyperplanes in $T_pM$ by 
\begin{equation}\label{Eq: vertical hyperplane in manifold}
	P_\alpha(p) := (D\phi)_0(P_\alpha^{m-1}\times 0^{L-m}), \quad \mbox{ for $\alpha\in [-\theta/2,\theta/2]$ and $P_\alpha^{m-1}$ given by (\ref{Eq: vertical hyperplane in wedge})}, 
\end{equation}
and define the set of {\em horizontal} hyperplanes of $T_pM$ (in $T_p\tM_{p,R}$) by 
\begin{equation}\label{Eq: horizontal hyperplane in manifold}
    \Pcal_{T_pM} := \{ (D\phi)_0(P\times 0^{L-m}) : P\in \Pcal_{T_0\Omega} \}.
\end{equation}
For any $P\in\Pcal_{T_pM}$, we call $P\cap T_pM$ a {\em horizontal wedge domain} in $T_pM$.

\subsection{Almost properly embedded locally wedge-shaped hypersurfaces}

Using the local model structure, we can define submanifolds with a certain regularity in locally wedge-shaped manifolds. 
\begin{definition}[Locally wedge-shaped submanifolds]\label{Def: locally wedge-shaped hypersurfaces}
	Let $N^n$ and $M^{n+k}$ be two locally wedge-shaped manifolds in $\R^L$ with $k\geq 0$. 
	Then we say $N^n$ is an embedded {\em locally wedge-shaped $n$-submanifold} of $M^{n+k}$, if 
	\begin{itemize}
		\item[(i)] $N\subset M$;
		\item[(ii)]for any $p\in N$, there exist local models $(\psi,\mB^L_R(p), \Omega_N)$ and $(\phi, \mB^L_R(p), \Omega_M)$ of $N$ and $M$ around $p$ respectively so that the associated local extension (given by (\ref{Eq: local extension})) $\widetilde{N}_{p,R}$ is embedded in $\tM_{p,R}$. 
	\end{itemize}
    Additionally, for $l\geq1$ and $\alpha\in (0,1]$, if $N$ is $C^{l,\alpha}$-to-edge, $\widetilde{N}_{p,R}$ is $C^{l,\alpha}$-embedded in $\tM_{p,R}$ and the embedding is $C^\infty$ away from $\bd^E N$, then we say $N$ is an {\em  $C^{l,\alpha}$-to-edge} embedded locally wedge-shaped $n$-submanifold.

    In particular, if $k=0$, then we say $N$ is a {\em locally wedge-shaped domain} of $M$; if $k=1$, then we say $N$ is a {\em ($C^{l,\alpha}$-to-edge) locally wedge-shaped hypersurface} of $M$.  
\end{definition}

Unless otherwise stated, the ambient manifold $M$ is always assumed to be $C^\infty$-to-edge, while the submanifold is $C^{1,1}$-to-edge. This is because the $C^{1,1}$-regularity is a necessary condition to get bounded curvature, and the optimal elliptic regularity in a general wedge domain is only $C^{1,\alpha}$ (cf. \cite{lieberman1989optimal}).

Using the stratification of locally wedge-shaped manifolds, we now introduce the definition of almost properly embedding, which is generalized from \cite[Definition 2.6]{li2021min}. 

\begin{definition}[Almost properly embedding]\label{Def: almost properly embedded}
	Let $M^{n+1}\subset \R^L$ be a locally wedge-shaped $(n+1)$-manifold, and $\Sigma^n\subset M$ be a locally wedge-shaped hypersurface in $M$. 
	We say $\Sigma$ is {\em almost properly embedded} in $M$, denoted by $(\Sigma,\{\partial_m\Sigma\})\subset (M, \{\partial_{m+1} M\})$, if 
	$$ \partial_m\Sigma \subset \partial_{m+1}M,\qquad \forall m\in\{n-2, n-1, n\}. $$
	In particular, if $\partial^F\Sigma = \partial^F M \cap \Sigma$ and $\partial^E\Sigma = \partial^E M\cap\Sigma$, then we say $(\Sigma,\{\partial_m\Sigma\})\subset (M, \{\partial_{m+1} M\})$ is {\em properly embedded}. 
\end{definition}

One can easily verify that a compact manifold with smooth boundary is a locally wedge-shaped manifold (modeled by trivial wedge domains), and thus the above definition is a generalization of \cite[Definition 2.6]{li2021min}. 
Additionally, we make the following remark:

\begin{remark}
	In \cite[Section 9]{edelen2022regularity}, Edelen-Li used $C^{1,\alpha}$-graphs over some horizontal hyperplane to define the regular set of a $1$-codimension integral current, which implies the $C^{1,\alpha}$-to-edge properly embeddedness in the sense of Definition \ref{Def: almost properly embedded}. 
\end{remark}

Although the free boundary area minimizers are proved to be {\em properly} embedded in \cite{edelen2022regularity}, we generally can not expect such properness to hold after taking limits due to the {\em touching phenomenon}. 
Namely, for $(\Sigma,\{\partial_m\Sigma\})\subset (M, \{\partial_{m+1} M\})$, $\interior(\Sigma)$ can touch $\bd^F M$ tangentially, and $\bd^F\Sigma$ can touch $\bd^E M$ tangentially. 
Therefore, to show a curvature estimate and the compactness theorem, we need to treat the touching set carefully. 

Specifically, consider an almost properly embedded locally wedge-shaped hypersurface $(\Sigma,\{\partial_m\Sigma\})\subset (M, \{\partial_{m+1} M\})$. 
It is simple to check that $T_p\Sigma$ is almost properly embedded in $T_pM$. 
Hence, we have 
\[\interior(\Sigma)\subset \interior(M)\cup \partial^FM \qquad {\rm and}\qquad \partial^F\Sigma\subset  \partial^F M\cup  \partial^EM. \]
We call 
\begin{itemize}
    \item $p\in \interior(\Sigma)\cap\bd^F M$ a {\em false boundary point} of $\Sigma$; and
    \item $p\in\bd^{F}\Sigma\cap\bd^E M$ a {\em false edge point} of $\Sigma$,
\end{itemize}
which together form the {\em touching set} of $\Sigma$. 

For simplicity, denote by 
\begin{equation}\label{Eq: boundary points on F&E}
	\partial^{FF}\Sigma := \partial^F\Sigma \cap \partial^FM \qquad{\rm and}\qquad \partial^{FE}\Sigma := \partial^F\Sigma\cap\partial^EM,
\end{equation}
the face of $\Sigma$ on $\partial^FM$ and $\partial^EM$ respectively. 
Then for any $p\in\Sigma\cap\partial^EM=\partial^E\Sigma\cup\partial^{FE}\Sigma$, 
\begin{eqnarray}\label{Eq: classification tangent cone}
	~~\qquad p\in \partial^E\Sigma &\Leftrightarrow & (T_p\Sigma,\{\partial_iT_p\Sigma\})\subset (T_pM,\{\partial_{i+1}T_pM\}) \mbox{ is a non-trivial wedge domain}; \nonumber
	\\ ~~\qquad p\in \partial^{FE}\Sigma &\Leftrightarrow & T_p\Sigma = P_\alpha(p) \mbox{ for some $\alpha\in [-\theta/2, \theta/2]$},
\end{eqnarray}
where $\theta=\theta(p)$ is the wedge angle of $M$ at $p$, and $P_\alpha(p)$ is given by (\ref{Eq: vertical hyperplane in manifold}).

\section{Variations in locally wedge-shaped manifolds}\label{Sec: variations in wedge manifold}

In the rest of this paper, we always denote by $(M^{n+1} , g_{_M})$ a ($C^\infty$-to-edge) locally wedge-shaped Riemannian $(n+1)$-manifold in some $\R^L$ with the induced metric. 

We say $X$ is a smooth vector field on $M^{n+1}$ if for any $p\in M$ there is a local model $(\phi, \mB^L_R(p), \Omega)$ and a smooth vector field $\widetilde{X} \in \mathfrak{X}\big((\R^{n+1}\times 0^{L-(n+1)}) \cap \mB^L_R(0)\big)$ so that 
\begin{itemize}
	\item $\widetilde{X}(x)\in T_x\Omega$ for all $x\in \Omega$,
	\item $X\llcorner (M\cap \mB^L_R(p)) = D\phi(\widetilde{X})\llcorner  (M\cap \mB^L_R(p))  $. 
\end{itemize}
Denote by $\mathfrak{X} (M)$ the set of smooth vector fields on $M$. 

Note $X\in \mathfrak{X} (M)$ if and only if for all $p\in M$, $X(p)\in T_pM$ and $X$ locally admits a smooth extension in $\widetilde{M}_{p,R}$. 
We define 
$$\mathfrak{X}_{\text{tan}} (M) :=\{ X\in \mathfrak{X}(M) : X(q) \in T_q\partial_i M, ~\forall p\in \partial_iM, ~i\in\{n-1,n,n+1\} \} $$
to be the set of vector fields $X\in \mathfrak X(M)$ whose flow preserves the stratification of $M$. 

\subsection{First variation for varifolds} 
Recall that the first variation of area for a varifold $V$ along the flow of a vector field $X$ is given by 
\begin{equation}\label{Eq: 1st variation for varifolds}
    \delta V(X) = \int \Div_{S} X(p) \, dV(p,S). 
\end{equation}

\begin{definition}\label{Def: stationary varifold}
    A varifold $V\in \mathcal V_n(M)$ is called {\it stationary with free boundary} if $\delta V(X) = 0$ for all $X\in \mathfrak X_{\text{tan}}(M)$. More generally, if $U$ is a relatively open subset of $M$ then $V$ is called {\it stationary with free boundary in} $U$ if $\delta V(X) = 0$ for all $X\in \mathfrak X_{\text{tan}}(M)$ with compact support in $U$. 
\end{definition}

Next, we record the monotonicity formula for stationary varifolds with free boundary in a locally wedge-shaped manifold $(M,\{\bd_{m+1} M\})$. 

\begin{proposition}[Monotonicity Formula]\label{Prop: monotonicity fomula} Let $(M,\{\bd_{m+1}M\})$ be a locally wedge-shaped manifold, and assume that $V\in \mathcal{V}_n(M)$ is stationary with free boundary in $M$. Then there are constants $C_M > 0$ and $r_0 > 0$, depending only on $M$, such that for all $p\in \bd^E M$ and all $0<s<t<r_0$ we have 
\[
    \frac{\|V\|(\mB^L_s(p))}{s^n} \le C_M\frac{\|V\|(\mB^L_t(p))}{t^n}.  
\]
\end{proposition}

The monotonicity formula is well-known at interior and faces points of $M$. The proof when $p\in \bd^E M$ uses arguments similar to those in Edelen-Li \cite{edelen2022regularity}.  For the sake of completeness, we include a detailed proof in Appendix \ref{Sec: monotonicity formula}. 
Additionally, it should be noted that the monotonicity formula is also valid for stationary varifolds in arbitrary codimension with free boundary.

\subsection{First variation for locally wedge-shaped hypersurfaces} 
We now return to the setting of almost properly embedded locally wedge-shaped hypersurfaces. We would like to introduce a variational problem for which:
\begin{itemize}
    \item 
    FBMHs $(\Sigma, \{\bd_{m} \Sigma\})$ properly embedded in $(M, \{\bd_{m+1} M\})$ are solutions;
    \item hypersurfaces $(\Sigma, \{\bd_m \Sigma\})$ for which $\bd \Sigma\subset \bd^E M$ but which do not meet either face of $M$ orthogonally are not solutions;
    \item
    the space of solutions is closed. 
\end{itemize}
This requires carefully designing a space of vector fields that are admissible for use in the first variation formula. 


\begin{definition}
    \label{admissible-vf}
    Let $(\Sigma,\{\bd_m \Sigma\})\subset  (M,\{\bd_{m+1}M\})$ be an almost properly embedded locally wedge-shaped hypersurface . We let $\mathfrak X(M,\Sigma)$ denote the space of vector fields $X\in \mathfrak X(M)$ such that 
    \begin{itemize}
        \item[(i)] $X(p)\in T_pM$ for all $p\in M$;
        \item[(ii)] for all $p\in \bd^E \Sigma$ there exists an $r > 0$ such that $X(q)\in T_q \bd^E M$ for all $q\in \mB^L_r(p) \cap \bd^E M$;
        \item[(iii)] for all $p\in \clos(\bd^{FF} \Sigma)$ there exists an $r>0$ such that $X(q)\in T_q\bd M$ for all $q\in \mB^L_r(p)\cap \bd M$;
        \item[(iv)] for all $p\in \bd^F\Sigma \setminus \clos(\bd^{FF}\Sigma)$ there exists $r > 0$ such that $\bd^F \Sigma \cap \mB^L_r(p) \subset \bd^E M$ and if $T_p\Sigma = P_\alpha(p)$ (see (\ref{Eq: classification tangent cone})) then $X\llcorner (\mB^L_r(p) \cap M) = \widetilde{X}\llcorner M$ for some
        $$
        \widetilde{X} \in
        \begin{cases}
            \mathfrak X_{\text{tan}}(\tM_{p,r}^+) & \text{if } \alpha < 0,\\
            \mathfrak X_{\text{tan}}(\tM_{p,r}^-) & \text{if } \alpha > 0,\\
            \mathfrak X_{\text{tan}}(\tM_{p,r}^+) \cup \mathfrak X_{\text{tan}}(\tM_{p,r}^-) &\text{if } \alpha = 0.
        \end{cases}
        $$
    \end{itemize}
\end{definition}

We make some comments on the definition. Condition (i) is imposed to ensure that the flow of $X$ sends $M$ into $M$. Conditions (ii) and (iii) are designed to ensure that the flow preserves the almost proper embeddedness of $(\Sigma,\{\bd_m \Sigma\})$. Note that admissible vector fields $X$ are allowed to push interior points of $\Sigma$ lying on a face of $M$ into the interior of $M$. Condition (iv) is designed to likewise allow admissible vector fields to push face points of $\Sigma$ lying on the edge of $M$  into a face of $M$. Figure \ref{fig:admissible-variations} illustrates the type of variations allowed by condition (iv). The reader should imagine that both pictures extend out of the page so that the face of $\Sigma$ is contained in the edge of $M$. Note how the flow of $X$ pushes the face of $\Sigma$ off the edge of $M$ and into a face of $M$, with the choice of face depending on the angle $\alpha$.

\begin{figure}
    \centering
    \includegraphics[width=4in]{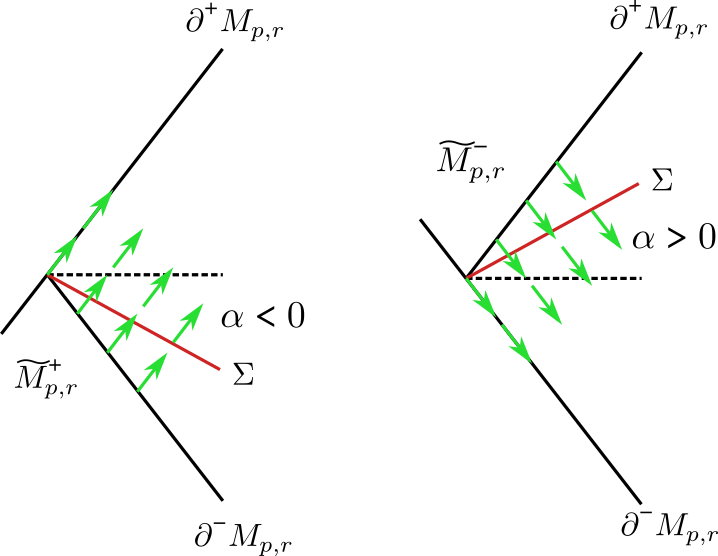}
    \caption{Admissible Variations}
    \label{fig:admissible-variations}
\end{figure}

\begin{remark}
The space $\mathfrak X(M,\Sigma)$ is always non-empty since it at least contains $\mathfrak X_{\text{tan}}(M)$. Also note that if $X \in \mathfrak X(M,\Sigma)$ then it is not in general true that $-X \in \mathfrak X(M,\Sigma)$. 


\end{remark}

\begin{proposition}
Assume that $(\Sigma,\{\bd_m\Sigma\})\subset (M,\{\bd_{m+1} M\})$ is almost properly embedded and that $X\in \mathfrak X(M,\Sigma)$. Let $\{f_t\}$ be the 1-parameter family of diffeomorphisms generated by $X$. Then $(f_t(\Sigma),\{\bd_m f_t(\Sigma)\}) \subset (M, \{\bd_{m+1}M\})$ is almost properly embedded for { $t\geq 0$} small enough. 
\end{proposition}

\begin{proof}
For all $t \ge 0$, it is clear that $f_t(\Sigma)$ is still a locally wedge-shaped manifold and that $\bd_m f_t(\Sigma) = f_t(\bd_m \Sigma)$ for $m = n-2,n-1,n$. Condition (i) implies that $f_t(M) \subset M$ for all $t \ge 0$, and thus $f_t(\Sigma)$ is still a locally wedge-shaped hypersurface in $M$. Condition (ii) ensures that  $f_t(\bd_{n-2} \Sigma) \subset \bd_{n-1} M$ for small enough $t \ge 0$. Condition (iii) ensures that $f_t(\clos(\bd^{FF} \Sigma)) \subset \bd_n M$ for $t\ge 0$ small enough. Finally for each $p\in \bd^F \Sigma \setminus \clos(\bd^{FF}\Sigma)$ there is a small number $r > 0$ for which $\bd^F\Sigma \cap \mB_r^L(p) \subset \bd^E M$.  Therefore conditions (i) and (iv) ensure that, shrinking $r$ if necessary, $f_t(\bd^F\Sigma \cap \mB_r^L(p))$ is either contained in $\bd^+ M_{p,r}$ or $\bd^- M_{p,r}$ for $t \ge 0$ sufficiently small. Combined with the previous observations, it follows that $\bd_{n-1}\Sigma \subset \bd_n M$ for sufficiently small $t\ge 0$. Thus $(f_t(\Sigma),\{\bd_m f_t(\Sigma)\})$ remains almost properly embedded for some short time. 
\end{proof}

Next, we consider the associated variational problem. According to the first variation formula for the mass of varifolds, one has 
\[
    \delta \Sigma(X) = \int_{\Sigma} \Div_{T_p\Sigma} X
\]
for any smooth vector field $X$. Since $\Sigma$ is a locally wedge-shaped manifold, Stokes' theorem implies that the first variation formula can be re-written as 
\begin{equation} 
    \label{smooth-fv}
    \delta \Sigma(X) = -\int_{\Sigma} \laa H_{\Sigma}, X\raa + \int_{\bd^F \Sigma} \laa \eta,X\raa,
\end{equation}
where $\eta$ denotes the unit outward co-normal along $\bd^F \Sigma$. 

\begin{definition}
    \label{definition:FBMH}
    We say that an almost properly embedded locally wedge-shaped hypersurface $(\Sigma,\{\bd_m\Sigma\})\subset (M,\{\bd_{m+1}M\})$ is a {\it free boundary minimal hypersurface} (abbreviated by FBMH) provided $\delta \Sigma(X) = 0$ for all $X\in \mathfrak X(M,\Sigma)$. 
\end{definition}

\begin{remark}
Note that the space of admissible variations depends on both $\Sigma$ and $M$.  
\end{remark}

Next, we characterize FBMHs $(\Sigma,\{\bd_m\Sigma\})\subset (M,\{\bd_{m+1}M\})$ in terms of their mean curvature and the manner in which they contact $\bd M$. 

\begin{proposition}\label{Prop: classify FBMH}
    An almost properly embedded hypersurface $(\Sigma^n,\{\bd_m\Sigma\})$ contained in $ (M^{n+1},\{\bd_{m+1}M\})$ is a FBMH if and only if the following three conditions hold:
    \begin{itemize}
        \item[(1)] $H_\Sigma \equiv 0$ in $\interior(\Sigma)=\bd_{n}\Sigma\setminus \bd_{n-1} \Sigma$, 
        \item[(2)] $\eta \perp \bd^F M$ on $\bd^{FF} \Sigma$, 
        \item[(3)] $\eta\perp \bd^+ M_{p,r}$ or $\eta\perp \bd^- M_{p,r}$ at any $p\in \bd^{FE}\Sigma$.
    \end{itemize}
\end{proposition}

\begin{proof}
First suppose that conditions (1)-(3) hold, and consider any $X\in \mathfrak X(M,\Sigma)$.   Condition (1) implies that 
\[
    \int_{\Sigma} \laa H_\Sigma,X\raa = 0.
\]
By condition (2) and part (iii) of Definition \ref{admissible-vf}, one has $\laa X,\eta\raa(p) = 0$ at all points $p\in \bd^F\Sigma \cap \Clos(\bd^{FF}\Sigma)$. Finally, condition (3) implies that at any point $p\in \bd^F\Sigma \setminus \clos(\bd^{FF}\Sigma)$ there is an $r>0$ and a choice of $\kappa\in \{\pm\}$ so that $\eta(q)\perp \bd^\kappa M_{p,r}$ for all $q\in \mB^L_r(p)\cap \bd^F \Sigma$.  Part (iv) of Definition \ref{admissible-vf} therefore implies that $\laa X, \eta\raa(p) = 0$ at any such $p$. Combining these observations, it follows that 
\[
    \int_{\bd^F \Sigma} \laa \eta,X\raa = 0,
\]
and therefore that $\delta \Sigma(X) = 0$. 

Conversely, suppose that $\delta \Sigma(X) = 0$ for all $X\in \mathfrak X(M,\Sigma)$. Since 
\[
(\bd_{n}\Sigma \setminus \bd_{n-1} \Sigma)\cap \bd^E M = \emptyset,
\] 
it follows easily from the definition of $\mathfrak X(M,\Sigma)$ that $H_\Sigma \equiv 0$ on $\bd_{n}\Sigma \setminus \bd_{n-1} \Sigma$. Next, note that $\bd^{FF} \Sigma$ is open in $\bd^F \Sigma$. Therefore, the first variation formula and the definition of $\mathfrak X(M,\Sigma)$ imply that $\eta \perp \bd^F M$ along $\bd^{FF}\Sigma$, and so condition (2) holds. Note that, by the assumed regularity of $\Sigma$, this also implies that condition (3) holds at any point $p\in \bd^{FE}\Sigma\cap \clos(\bd^{FF}\Sigma)$.  

Now observe that the set 
\[
    S = \{p\in \bd^{FE} \Sigma\setminus \clos(\bd^{FF}\Sigma):\, T_p\Sigma = P_\alpha(p) \text{ with } \alpha\neq 0\}
\]
is also open in $\bd^F\Sigma$. By the first variation formula and part (iv) of Definition \ref{admissible-vf}, it follows that condition (3) holds at all $p\in S$.  It remains to consider the set
\[
    S_0 = \{p\in \bd^{FE}\Sigma \setminus \clos(\bd^{FF}\Sigma): T_p\Sigma = P_\alpha(p) \text{ with } \alpha = 0\}.
\]
Note that $S_0$ is closed in $\bd \Sigma$. Additionally, $\eta$ is not perpendicular to either face of $M$ along $S_0$.  By the first variation formula and part (iv) of Definition \ref{admissible-vf}, it follows that the set $S_0$ has an empty interior in $\bd^F \Sigma$. In particular, this implies that $S_0 \subset \clos(\bd^F \Sigma\setminus S_0)$. Since we have already seen that $\eta$ is perpendicular to at least one face of $M$ at all points of $\bd^F\Sigma \setminus S_0$, the smoothness of $\Sigma$ implies that in fact $S_0 = \emptyset$.  Thus conditions (1)-(3) are satisfied. 
\end{proof}

\begin{remark}\label{Rem: orthogonally meeting condition}
    Since $\theta(p)<\pi$ for $p\in\bd^EM$, the orthogonally meeting conditions (2)(3) in Proposition \ref{Prop: classify FBMH} indicate that for any $p\in\bd^{FE}\Sigma$, there exists $r>0$ so that $\bd\Sigma\cap\mB^L_r(p)$ is contained either in $\bd^+M_{p,r}$ or $\bd^-M_{p,r}$, and $T_p\Sigma = P_{\pm(\theta(p)-\pi)/2}$ (see (\ref{Eq: classification tangent cone})). 
    Additionally, if $p\in \bd^E\Sigma$, then (2) implies $T_p\Sigma=P\cap T_pM$ for some horizontal hyperplane $P\in\Pcal_{T_pM}$.
\end{remark}

\subsection{Second variation for locally wedge-shaped hypersurfaces} 

Next, we investigate the stability of locally wedge-shaped hypersurfaces. Assume that $(\Sigma,\{\bd_m\Sigma\})\subset (M,\{\bd_{m+1}M\})$ is an almost properly embedded locally wedge-shaped hypersurface. Suppose that $\Sigma$ is two-sided so that it admits a continuous choice of normal vector $\nu$. 

\begin{proposition}
    Assume that $\Sigma$ is a FBMH.  Suppose that $X\in \mathfrak X(M,\Sigma)$ restricts to the vector field $f\nu$ on $\Sigma$. Then the second variation of area is given by 
    \[
        \delta^2\Sigma(X) = \int_{\Sigma} \vert \nabla f\vert^2 - (\vert A\vert^2 - \Ric^M(\nu,\nu))f^2  + \int_{\bd \Sigma} \laa \nabla_X X, \eta \raa
    \]
    where $A$ denotes the second fundamental form of $\Sigma$ and $\eta$ is the unit outward co-normal along $\bd \Sigma$.
\end{proposition}

\begin{proof}
    In general, the second variation is given by 
    \begin{align*}
    \delta^2 \Sigma(X) &= \int_{\Sigma} \Div_{T_p\Sigma}(\nabla_X X) + (\div_{T_p\Sigma} X)^2 + \vert \del^\perp X\vert^2 \\
    &\qquad - \int_{\Sigma} \sum_{i=1}^n R^M(X,e_i,X,e_i) - \sum_{i,j=1}^n \laa \del_{e_i} X,e_j\raa \laa \del_{e_j} X, e_i\raa.
    \end{align*}
    See Section 9 of \cite{simon1983lectures} for a proof when the ambient space is Euclidean. It is straightforward to compute that $\vert \del^\perp X\vert^2 = \vert \grad f\vert^2$. Since $\Sigma$ is a FBMH and $X$ is normal, we have $\div_{T_p\Sigma} X \equiv 0$. Again using that $X$ is normal, we have 
    \[
    \sum_{i,j=1}^n \laa \del_{e_i} X,e_j\raa \laa \del_{e_j} X, e_i\raa = f^2 \vert A\vert^2, \quad \text{and}\quad 
    \sum_{i=1}^n R^M(X,e_i,X,e_i) = f^2 \Ric^M(\nu,\nu). 
    \]
    Finally, since $\Sigma$ is a locally wedge-shaped manifold, we can integrate by parts to get 
    \[
    \int_{\Sigma} \div_{T_pM}(\del_X X) = \int_{\bd \Sigma} \laa \del_X X, \eta\raa. 
    \]
    Combining all of the above observations gives the desired formula. 
\end{proof}

\begin{definition}
    \label{definition:stability}
    Assume $(\Sigma,\{\bd_m\Sigma\}) \subset (M,\{\bd_{m+1}M\})$ is an almost properly embedded locally wedge-shaped hypersurface which is minimal with free boundary. Then we say $\Sigma$ is {\it stable} provided $\Sigma$ is two-sided and $\delta^2\Sigma(X) \ge 0$ for all $X = f\nu \in \mathfrak X(M,\Sigma)$.  If $U$ is a relatively open subset of $M$, then $\Sigma$ is said to be {\it stable in $U$} provided $\delta^2\Sigma(X) \ge 0$ for all $X\in \mathfrak X(M,\Sigma)$ with compact support in $U$ and $\Sigma$ is $2$-sided in $U$.  
\end{definition}

\section{Convergence of locally wedge-shaped free boundary minimal hypersurfaces}
\label{sec:compactness}

In this section, we study the convergence of locally wedge-shaped FBMHs with uniformly bounded area and second fundamental form. 
Since the smooth convergence in $M\setminus\bd^E M$ has been investigated by Guang-Li-Zhou \cite{guang2020curvature}, we mainly focus on the points in $\bd^E M$. 
Although these results share the same spirit as \cite[Section 6]{guang2020curvature}, the corresponding discussion will be more complicated due to the lack of Fermi coordinates. 

To begin with, we review some basic facts on the local Riemannian geometry of a locally wedge-shaped manifold. 

\begin{lemma}\label{Lem: local chart}
    Let $M^{n+1}\subset \R^L$ be a locally wedge-shaped Riemannian manifold (with induced metric). 
    Given $\epsilon > 0$, there exist $r_0 = r_0(\epsilon, M) >0$ and $C_0=C_0(M)>0$ so that for any $p\in\bd^E M$ there is a local model $(\phi_p, \mB_{r_0}^L(p), \Omega_p)$ of $M$ around $p$ satisfying
    \begin{enumerate}[label=(\roman*)]
        \item $\frac{1}{1+\epsilon}\leq |D\hat{\phi}_p| \leq 1+\epsilon$ in $\Omega_{r_0} := \Omega_p\cap \mB^{n+1}_{r_0}(0)$, 
        \item $\|g_{_{\R^{n+1}}} - \hat{\phi}_p^*g_{_M} \|_{C^0(\Omega_{r_0})} \leq \epsilon$, and $\|g_{_{\R^{n+1}}} - \hat{\phi}_p^*g_{_M} \|_{C^2(\Omega_{r_0})}\leq C_0$,
        \item $|A_{\bd^E M}|\leq C_0$ in $\mB^{L}_{r_0}(p)\cap \bd^E M$, 
        \item $|A_{\bd^F M}|\leq C_0$ in $\mB^{L}_{r_0}(p)\cap \bd^F M$,
    \end{enumerate}
    where $\hat{\phi}_p := \phi_p \llcorner (\mB^{n+1}_{r_0}(0)\times 0^{L-(n+1)})$, $A_{\bd^E M}$ and $A_{\bd^F M}$ are the second fundamental forms of $\bd^E M$ and $\bd^F M$ respectively. 
\end{lemma}
\begin{proof}
    For any $p\in\bd^E M$, since $g_{_M}$ is induced from $g_{_{\R^L}}$ and $(D\phi_p)_0\in O(L)$ (by Definition \ref{Def: wedge manifold}(i)), we can take $r_0>0$ small enough so that $\frac{1}{1+\epsilon}\leq |D\hat{\phi}| \leq 1+\epsilon$ in $\Omega_{r_0} := \Omega\cap \mB^{n+1}_{r_0}(0)$, and $\|g_{_{\R^{n+1}}} - \hat{\phi}^*g_{_M} \|_{C^0(\Omega_{r_0})} \leq \epsilon$. 
    Note $\bd^E M$ is a closed smooth $(n-1)$-manifold (\ref{Eq: edge is closed manifold}). 
    The rest part follows directly from a standard compactness argument. 
\end{proof}

Take $\epsilon>0$ small enough so that $(1+\epsilon)^4 < \min\{\frac{\pi}{2}, \frac{32}{27}\}$ and $[\cos((1+\epsilon)^4/2)]^{-2} < \frac{3}{2}$. 
Then for the corresponding $r_0>0$ and the local models $\{(\phi_p, \mB^L_{r_0}(p), \Omega_p)\}_{p\in\bd^EM}$ given by Lemma \ref{Lem: local chart}, we have the following result generalizing \cite[Lemma 2.4]{colding2011course}. 

\begin{lemma}[Small curvature implies graphical]\label{Lem: small curvature implies graphical}
    Suppose $(\Sigma, \{\bd_m \Sigma\})\subset (M, \{\bd_{m+1} M\})$ is an almost properly embedded $C^{1,1}$-to-edge locally wedge-shaped hypersurface satisfying
    \begin{itemize}
        \item[(1)] $\Sigma$ meets $\bd M$ orthogonally along $\bd \Sigma$ in the sense of Proposition \ref{Prop: classify FBMH}(2)(3),
        \item[(2)] $\sup_\Sigma|A_\Sigma|^2 \leq \frac{1}{16s^2}$ for some $s\in (0, \min\{\frac{r_0}{4}, \frac{1}{8}\})$,
    \end{itemize}
    where $A_\Sigma$ is the second fundamental form of $\Sigma$.
    Then using the notations in Lemma \ref{Lem: local chart}, for any $p_0\in \bd^EM$ and $ p \in\Sigma\cap \mB^L_{r_0/2}(p_0)$, we have
    \begin{itemize}
        \item[(i)] $(\hat{\phi}_{p_0})^{-1}(B^\Sigma_{2s_0}(p))$ can be written as a graph of a  $C^{1,1}$-function $\hat{u}$ over its tangent space at $x_0=(\hat{\phi}_{p_0})^{-1}(p)$ so that $\hat{u}$ is smooth away from the edge, $|\nabla \hat{u}|\leq 1$ and $|\Hess(\hat{u})|\leq \frac{1+C_0}{\sqrt{2}s}$ with respect to the metric $\hat{g} := \hat{\phi}_{p_0}^{*}g_{_M}$;
        \item[(ii)] the connected component $\Sigma'$ of $\mB^L_{s_0}(p)\cap\Sigma$ containing $p$ is contained in $B^\Sigma_{2s_0}(p)$, 
    \end{itemize}
    where $s_0:=\frac{s}{1+C_0}$, and $B^\Sigma_{t}(p)$ is the intrinsic ball in $\Sigma$ centered at $p$ with radius $t>0$. 
\end{lemma}
\begin{proof}
    Note $B^\Sigma_{2s}(p)\subset \mB^L_{r_0/2}(p)\cap M \subset \mB^L_{r_0}(p_0)\cap M$. 
    We denote by $\widehat{\Sigma} := (\hat{\phi}_{p_0})^{-1}(\Sigma\cap \mB^L_{r_0/2}(p))$, and pull back all the computations into $\Omega_{p_0}\cap \mB^{n+1}_{r_0}(0)$. 
    Let $g_{_0}:= g_{_{\R^{n+1}}}$. By Lemma \ref{Lem: local chart}(ii), 
    $$ \sup |A_{\widehat{\Sigma}, g_{_0}}| ~\leq ~ (1+\epsilon)^2(1+C_0)\sup |A_{\widehat{\Sigma}, \hat{g}}| ~\leq ~ \frac{(1+\epsilon)^2(1+C_0)}{4s} .$$
    Let $N:\widehat{\Sigma} \to \mS^n $ be the Gauss map (with respect to $g_0$), and $d_{x_0,y} := \dist_{\mS^n}(N(x), N(y))$ for $x_0:=(\hat{\phi}_{p_0})^{-1}(p),y\in \widehat{\Sigma}$. 
    Given  $y\in\widehat{\Sigma}$, consider the piecewise smooth curve $\alpha$ connecting $x_0$ and $y$ in $\widehat{\Sigma}$ so that $L(\alpha,g_{_0})\leq (1+\epsilon)\dist_{\widehat{\Sigma},g_{_0}}(x_0,y)$, where $L(\alpha,g_{_0})$ is the length of $\alpha$ with respect to $g_{_0}$. 
    Combining the computations in \cite[(2.45)]{colding2011course} with Lemma \ref{Lem: local chart}(i), we obtain
    $$d_{x,y}\leq \frac{(1+\epsilon)^3(1+C_0)}{4s}\dist_{\widehat{\Sigma}, g_{_0}}(x_0,y)\leq \frac{(1+\epsilon)^4(1+C_0)}{4s}\dist_{\widehat{\Sigma}, \hat{g}}(x_0,y) .$$
    Hence, for any $q\in B^{\Sigma}_{2s_0}(p)$, we have $y=(\hat{\phi}_{p_0})^{-1}(q)\in B^{\widehat{\Sigma},\hat{g}}_{2s_0}(x_0)$ and $d_{x_0,y} \leq \frac{(1+\epsilon)^4}{2} < \frac{\pi}{4}$ by the choice of $\epsilon>0$, where $B^{\widehat{\Sigma},\hat{g}}_{t}(x_0)$ is the intrinsic ball of $\widehat{\Sigma}$ under the metric $\hat{g}$. 
    Hence, $\widehat{\Sigma}$ is a graph of function $\hat{u}$ over $T_{x_0}\widehat{\Sigma}$. 
    Additionally, the computations in \cite[Lemma 2.4]{colding2011course} give the following estimates (away from the edge):
    $$ |\nabla^{g_{_0}} \hat{u}|\leq \left([\cos((1+\epsilon)^4/2)]^{-2} - 1 \right)^{\frac{1}{2}} \leq \frac{\sqrt{2}}{2}, \quad |\Hess_{g_{_0}}\hat{u}|^2\leq \frac{27}{8}\sup |A_{\widehat{\Sigma}, g_{_0}}|^2 < \frac{(1+C_0)^2}{4s^2} ,$$
    by the choice of $\epsilon$. 
    Therefore, combining this with Lemma \ref{Lem: local chart}(ii) gives the estimates of $\hat{u}$ with respect to the metric $\hat{g}$. 
    
    For any $q\in\Sigma$ with $\dist_\Sigma(p,q) = 2s_0$, there is a piecewise smooth curve $\gamma:[0,l]\to\Sigma$ so that $2s_0\leq l=L(\gamma)\leq (1+\epsilon)2s_0$, and 
    $ \gamma\llcorner (t_i,t_{i+1})$ is a geodesic in $\interior(\Sigma)$ or $\bd^F \Sigma$ or $\bd^E\Sigma$, where $0=t_0<t_1<\dots<t_k=l$. 
    Note the geodesic exists in $\bd^E\Sigma$ since $\Sigma$ is $C^{1,1}$-to-edge.
    It follows from the almost properly embeddedness and condition (1) that 
    $$ |\nabla^M_{\gamma'}\gamma'| 
        \leq \left\{ \begin{array}{ll}
		          |A_{\Sigma}(\gamma',\gamma')|, ~&~{\rm on~} \gamma\llcorner (t_i,t_{i+1}) \subset \interior(\Sigma) ,\\
		          |A_{\bd^F M}(\gamma',\gamma')| +|A_{\Sigma}(\gamma',\gamma')|,~ &~{\rm on~}\gamma\llcorner (t_i,t_{i+1}) \subset \bd^F \Sigma, \\
		          |A_{\bd^E M}(\gamma',\gamma')| +|A_{\Sigma}(\gamma',\gamma')|, ~&~{\rm on~} \gamma\llcorner (t_i,t_{i+1}) \subset \bd^E \Sigma~{{\rm (almost~everywhere)}},
	\end{array} \right.  $$
    which further implies $|\frac{d}{dt} g_{_{\R^L}}(\gamma'(t),\gamma'(0))| = |g_{_M}(\nabla^M_{\gamma'}\gamma'(t), \gamma'(0))|\leq C_0 + \frac{1}{4s}$ by condition (2) and Lemma \ref{Lem: local chart}(iii)(iv). 
    After integrating this from $0$ to $t\in [0,l]$, we see 
    $g_{_{\R^L}}(\gamma'(t),\gamma'(0))\geq 1-C_0t-\frac{t}{4s}$. 
    Therefore, since $2s<1/4$, we obtain
    \begin{eqnarray*}
        |q-p|\geq g_{_{\R^L}}(q-p, \gamma'(0)) &\geq& \int_0^l 1-C_0t-\frac{t}{4s} ~dt ~=~ l- \frac{C_0}{2}l^2 - \frac{1}{8s}l^2\\
        &\geq& s_0\left[ 2 - 2C_0(1+\epsilon)^2s_0 - \frac{(1+\epsilon)^2}{2s}s_0 \right]\\
        &\geq& \frac{s}{1+C_0}\left(2-\frac{3(1+\epsilon)^2}{4}\right) > \frac{s}{1+C_0},
    \end{eqnarray*}
    and $q$ lies outside of $\Sigma\cap\mB^L_{s_0}(p)$, which implies (ii).  
\end{proof}

\begin{lemma}\label{Lem: properness of graph}
    Using the assumptions and notations in Lemma \ref{Lem: small curvature implies graphical}, we also have the following results by shrinking $r_0>0$ even smaller (independent on $\Sigma$):
    \begin{itemize}
        \item[(i)] if $p\in \bd^E\Sigma'$ then $\Sigma'$ is properly embedded in $M_{p_0,r_0}$;
        \item[(ii)] if $p\in\bd^{FE}\Sigma'$, then $\Sigma'$ is properly embedded either in $\tM^{+}_{p_0,r_0}$ or $\tM^{-}_{p_0,r_0}$; 
        \item[(iii)] if $p\in\bd^{FF}\Sigma'$ with $T_p\Sigma'$ sufficiently close to the vertical half-plane $P_{\kappa(\theta(p_0)-\pi)/2}\subset T_{p_0}M$ defined in (\ref{Eq: vertical hyperplane in manifold}), then $\Sigma'$ is properly embedded in $\tM^{\ka}_{p_0,r_0}$ for some $\kappa\in\{\pm\}$. 
    \end{itemize} 
\end{lemma}
\begin{proof}
    As in the proof of Lemma \ref{Lem: small curvature implies graphical}, let $\widehat{\Sigma}':=(\hat{\phi}_{p_0})^{-1}(\Sigma')$, and $N:\widehat{\Sigma}'\to \mS^{n+1}$ be the Gauss map, which satisfies $d_{x_0, y}<\frac{\pi}{4}$ for $x_0:=(\hat{\phi}_{p_0})^{-1}(p),y\in \widehat{\Sigma}'$. 
    Recall also that $\hat{g}$ is sufficiently $C^0$-close to the Euclidean metric $g_0$ (by shrinking $r_0>0$). 

    For the case $p\in\bd^E\Sigma'$, suppose there exists $q=\hat{\phi}_{p_0}(y)\in \bd^{FE}\Sigma'$ or $q=\hat{\phi}_{p_0}(z)\in \interior(\Sigma')\cap\bd^F M$. 
    Then by the orthogonally meeting condition and $\|\hat{g} - g_0\|_{C^0}<\epsilon$, $N(x_0)$ is sufficiently close to $\mS^{n-1}:=\{v\in T_{x_0}\bd^E\Omega: |v|_{g_0}=1\}=\{(0,0,x_3,\dots,x_{n+1})\in \mS^{n+1}\}$, while $N(y)$ and $N(z)$ are sufficiently close to $\{(x_1,x_2,0,\dots,0)\in\mS^{n+1}\}$ making $d_{x_0,y}$ and $d_{x_0,z}$ close to $\frac{\pi}{2}$, which is a contradiction. 

    For the case $p\in\bd^{FE}\Sigma'$, assume without loss of generality that a neighborhood of $p$ in $\bd^{F}\Sigma'$ contained in $\bd^+M_{p_0,r_0}$ (by Remark \ref{Rem: orthogonally meeting condition}), and thus $N(x_0) \approx (\cos(\frac{\theta_0}{2}), \sin(\frac{\theta_0}{2}),0,\dots,0)$.  
    Noting $d_{x_0, y}<\frac{\pi}{4}$ and $\theta(p)\geq \frac{\pi}{2}$ by assumptions, an argument similar to the above implies $\bd^E\Sigma'=\emptyset$, $\bd^F\Sigma'\subset \bd^+M_{p_0,r_0} $, and $\interior(\Sigma')\cap \bd^+M_{p_0,r_0} =\emptyset$, which gives (ii). 
    Finally, (iii) holds for the same reasons as (ii). 
\end{proof}

Before we state the main convergence result of this section, let us take a closer look at the free boundary minimal graph function $\hat{u}$ to facilitate our discussion. 
Due to the lack of Fermi coordinates, the orthogonally meeting boundary conditions of $\hat{u}$ are not as clean as represented in \cite[(6.1)]{guang2020curvature}. 
Nevertheless, $\hat{u}$ still satisfies the so-called {\em oblique derivative} boundary conditions (c.f. \cite{lieberman1988oblique}). 
Specifically, consider a FBMH $(\Sigma, \{\bd_m \Sigma\})\subset (M, \{\bd_{m+1} M\})$ and a point $p_0\in\bd^EM$ so that $\widehat{\Sigma}:=\hat{\phi}_{p_0}(\Sigma\cap\mB^L_{r}(p_0))$ is a graph of a function $\hat{u}$ over a domain in the plane $P\subset T_0\R^{n+1}$, where $(\phi_{p_0}, \mB^L_{r_0}(p_0), \Omega_{p_0})$ is a local model of $M$ near $p_0$. 
We still denote by $\hat{g}:=\hat{\phi}_{p_0}^*g_{_M}$ and $\Omega_{r}:=\Omega_{p_0}\cap\mB^{n+1}_{r}(0)$.

Firstly, suppose $P\in\Pcal_{\Omega}$ is a horizontal hyperplane (\ref{Eq: horizontal plane in domain}) and $\Sigma\cap \mB^L_{r}(p_0)$ is properly embedded in $M\cap \mB_r^L(p_0)$. 
By a further rotation, we may assume $P=\{x_{n+1}=0\}$ and
\[{\rm Graph}(\hat{u}) = \{(x_1,\dots,x_n, \hat{u}(x_1,\dots,x_n))\}.\]
Note $\hat{u}$ is defined on $ P\cap\Omega_r$ and $\bd^{\pm}\widehat{\Sigma} = {\rm Graph}(\hat{u}\llcorner(P\cap \bd^{\pm}\Omega_r))$ for $r>0$ small enough. 
By a computation similar to \cite[(7.11)]{colding2011course}, the upward pointing unit normal of ${\rm Graph}(\hat{u})$ is
\[\nu_{\widehat{\Sigma}}(x',\hat{u}(x')) = \frac{-\hat{u}_i\hat{g}^{ia} + \hat{g}^{n+1,a}}{\sqrt{\hat{g}^{n+1,n+1}-2\hat{u}_i\hat{g}^{i,n+1} + \hat{u}_i\hat{u}_j\hat{g}^{ij}}} (x',\hat{u}(x')) \frac{\partial}{\partial x_a}\Big|_{(x',\hat{u}(x'))} ,\]
where the summation convention is used, $i,j\in\{1,\dots,n\}$, $a,b\in\{1,\dots,n+1\}$, $x'=(x_1,\dots,x_n)$, $\hat{u}_i:=\frac{\partial \hat{u}}{\partial x_i}$, $\hat{g}_{ab}$ and $\hat{g}^{ab}$ are the representation of $\hat{g}$ and its inverse in the coordinates. 
Similarly, the inward unit normal of $\bd^{\pm}\Omega \subset \{x_1= \pm\cot(\theta_0/2)x_2\}$ is 
\[ -\nu_{\pm}  
~=~ \frac{\mp\cot(\frac{\theta_0}{2})\hat{g}^{2a} + \hat{g}^{1a}}{\sqrt{\hat{g}^{11} \mp 2\cot(\frac{\theta_0}{2})\hat{g}^{12} + \cot^2(\frac{\theta_0}{2})\hat{g}^{22}}} \frac{\partial}{\partial x_a} ,\]
where $\theta_0 := \theta(p_0)\in (0,\pi)$. 
Hence, the orthogonally meeting boundary condition implies 
\begin{equation}\label{Eq: boundary condition of FBMH on horizontal plane}
    \sum_{i=1}^n\beta_{\pm,\hat{u}}^i (x') \cdot \hat{u}_i(x')  = \beta_{\pm,\hat{u}}^{n+1}(x'), \qquad \forall x'\in P\cap \bd^{\pm}\Omega_r=\{x_{n+1}=0\}\cap \bd^{\pm}\Omega_r,
\end{equation}
where $\beta_{\pm,\hat{u}}^a(x'):=(\mp\cot(\frac{\theta_0}{2})\hat{g}^{2a} + \hat{g}^{1a})(x',\hat{u}(x'))$.

On the other hand, if $\theta_0(p_0)\geq \frac{\pi}{2}$, $P_{(\theta(p_0)-\pi)/2}\subset P$ is the vertical hyperplane (see (\ref{Eq: vertical hyperplane in wedge})) and $\Sigma\cap \mB^L_{r}(p_0)$ is properly embedded in $\tM^{+}_{p_0,r}$. 
Then, by an orthogonal linear transformation, we transform $\Omega$, $P$, and $\bd H_+$ into $\{x_{n+1}\geq \cot(\theta(p_0)) x_1, ~x_1\geq 0\}$, $\{x_{n+1}=0\}$, and $\{x_1 = 0\}$ respectively. 
Hence, $\widehat{\Sigma}={\rm Graph}(\hat{u})=\{(x_1,x_2,\dots,\hat{u}(x_1,\dots,x_{n}))\}$ again. 
Similarly, we see $\hat{u}$ is defined on $P\cap\Omega_r$ and $\bd\widehat{\Sigma}={\rm Graph}(\hat{u}\llcorner(P\cap\bd H_+))$ for $r>0$ small enough. 
As before, we get the boundary condition:
\begin{equation}\label{Eq: boundary condition of FBMH on vertical plane}
    \sum_{i=1}^{n} \beta^i_{+,\hat{u}} (x')\cdot \hat{u}_i(x') = \beta^{n+1}_{+,\hat{u}}(x'), \qquad \forall x'\in P\cap \bd H_+ = \{x_1=x_{n+1}=0 \},
\end{equation}
where $\beta^a_{+,\hat{u}} (x') := g^{1a}(\hat{u}(x'),x')$, $a=1,\dots,n+1$, and $x'=(x_1,\dots,x_{n})$. 



We now show the main result of this section, which is a convergence theorem for almost properly embedded locally wedge-shaped FBMHs with uniformly bounded area and the second fundamental form. 

\begin{theorem}\label{Thm: curvature bound implies convergence}
    Let $(\Sigma_i^n,\{\bd_m\Sigma_i\})\subset(M^{n+1},\{\bd_{m+1}M\})$ be a sequence of almost properly embedded, $C^{1,1}$-to-edge locally wedge-shaped, FBMHs so that 
    \[ \sup_{i}{\rm Area}(\Sigma_i)\leq C \qquad {\rm and} \qquad \sup_{i}\sup_{\Sigma_i}|A_{\Sigma_i}|\leq C \]
    for some constant $C>0$. 
    Then after passing to a subsequence, $(\Sigma_i,\{\bd_m\Sigma_i\})$ {converges $C^{1,\alpha}$-to-edge ($\forall \alpha\in (0,1)$), $C^\infty$ away from the edge, locally uniformly} to an almost properly embedded {$C^{1,1}$-to-edge} locally wedge-shaped hypersurface $(\Sigma,\{\bd_m\Sigma\})\subset(M,\{\bd_{m+1}M\})$ (with multiplicity) which is also minimal in $M$ with free boundary. 
\end{theorem}
\begin{proof}
    Let $V\in \V_n(M)$ be the varifolds limit of $|\Sigma_i|$ (up tp a subsequence), and $\Sigma := \spt(\|V\|)$. 
    By the monotonicity formula (Proposition \ref{Prop: monotonicity fomula}), the convergence is also in the Hausdorff distance, and thus $\lim_{i\to\infty}\bd^E\Sigma_i \subset \bd^EM$. 
    Additionally, \cite[Theorem 6.1]{guang2020curvature} gives the smoothly convergence and the regularity of $\Sigma$ in $M\setminus\bd^E M$. We now consider the convergence near any $p_0\in\Sigma\cap\bd^EM$. 

    Use the notations $r_0>0$, $(\phi_{p_0}, \mB^L_{r_0}(p_0), \Omega_{p_0})$, $\Omega_{r}:=\Omega_{p_0}\cap\mB^{n+1}_{r}(0)$, $\hat{g}:=\hat{\phi}_{p_0}^*g_{_M}$ and $s_0:=\frac{s}{1+C_0}$ as in Lemma \ref{Lem: small curvature implies graphical}, where $s>0$ is small enough so that $C<\frac{1}{4s}$ and $s<\min\{\frac{r_0}{4}, \frac{1}{8}\}$. 
    Let $\{\Sigma_i^{j}\}_{j=1}^{k_i}$ be the connected component of $\Sigma_i\cap \mB^L_{s_0/2}(p_0)$. 
    By the monotonicity formula (Proposition \ref{Prop: monotonicity fomula}), we see $k_i$ is uniformly bounded, and thus $k_i\equiv k$ up to a subsequence. 
    Additionally, the Hausdorff topology convergence indicates the existence of $p_i^{1}\in \Sigma_i$ with $\lim_{i\to\infty}p_i^{1} = p_0$. 
    In particular, we can assume without loss of generality that $p_i^1\in \Sigma_i^1$, $|p_i^1 - p_0|<\frac{s_0}{2}$, and either
    \begin{itemize}
        \item[(1)] $p_i^1\in \interior(\Sigma_i)$ for all $i\geq 1$; or
        \item[(2)] $p_i^1\in \bd^F\Sigma_i\cap \bd^{\kappa} M_{p_0,r_0}$ for all $i\geq 1$, where $\kappa\in\{\pm\}$; or
        \item[(3)] $p_i^1\in \bd^E\Sigma_i$ for all $i\geq 1$.
    \end{itemize}
    Hence, $\Sigma_i^1$ is contained in the connected component of $\Sigma_i\cap \mB^L_{s_0}(p_i^1)$ containing $p_i^1$. 
    It then follows from Lemma \ref{Lem: small curvature implies graphical} that $\widehat{\Sigma}_i^1 := (\hat{\phi}_{p_0})^{-1}(\Sigma_i^1)$ is a graph of function $\hat{u}_i^1$ on $T_{x_i^1}\widehat{\Sigma}_i^1$ so that {$\|\hat{u}_i^1\|_{C^{1,1}}$} is uniformly bounded, where $x_i^1:=(\hat{\phi}_{p_0})^{-1}(p_i^1)$. 
    Denote by $P_i$ the hyperplane containing $T_{x_i^1}\widehat{\Sigma}_i^1$, which tends to a hyperplane $P$ at $0$ (up to a subsequence). 
    Now we separate the discussion into the above three cases. 

    {\bf Case (1):} $p_i^1\in \interior(\Sigma_i)$ for all $i\geq 1$. 
    Then $P_i = T_{x_i^1}\widehat{\Sigma}_i^1\cong \R^n \to P$ as $i\to\infty$. 
    Since $P$ is transversal to at least one of $\bd^{\pm}\Omega$ at $0$, $\interior(\widehat{\Sigma}_i^1)$ contains points outside of $\Omega$ for $i$ large enough, which is a contradiction. 

    {\bf Case (2):} $p_i^1\in \bd^F\Sigma_i\cap \bd^{\kappa} M_{p_0,r_0}$ for all $j\geq 1$, where $\kappa\in\{\pm\}$. Assume $\kappa=+$ for simplicity. 
    Then in the coordinates introduced before (\ref{Eq: boundary condition of FBMH on vertical plane}), $T_{x_i^1}\widehat{\Sigma}_i^1 = P_i\cap\{x_1\geq 0\}$ is a half-plane normal to $\bd^+\Omega$ at $x_i^1$ (under the metric $\hat{g}$). 
    Thus, $P$ is a hyperplane normal to $\bd^+\Omega$ at $0$. 
    For the same reason as in Case (1), we must have $P=\{x_{n+1}=0\}$ and $\theta(p_0)\geq \frac{\pi}{2}$. 
    By Lemma \ref{Lem: properness of graph}(ii)(iii), 
    we can apply the standard elliptic PDE theory with oblique derivative boundary conditions (\ref{Eq: boundary condition of FBMH on vertical plane}) in an $r$-neighborhood of $0\in\{x_1\geq 0\}$, and see ${\rm Graph}(\hat{u}_i^1)$ converges smoothly and graphically to a FBMH {\em without edge} $({\rm Graph}(\hat{u}^1), \bd {\rm Graph}(\hat{u}^1))\subset (\{x_{n+1}\geq 0\}, \{x_{n+1}=0\})$, for some $\hat{u}^1\in C^{\infty}(P\cap \{x_1\geq 0\}\cap \mB_r(0))$. 
    Since $({\rm Graph}(\hat{u}_i^1),\bd {\rm Graph}(\hat{u}_i^1))\subset(\Omega, \bd^+\Omega)$, we also have $({\rm Graph}(\hat{u}^1),\bd {\rm Graph}(\hat{u}^1))\subset(\Omega, \bd^+\Omega)$, whose interior may touch with $\bd^-\Omega$ from inside. 

    {\bf Case (3):} $p_i^1\in \bd^E\Sigma_i$ for all $i\geq 1$. 
    Now, $P_i$ is a hyperplane perpendicular to $\bd^{\pm}\Omega$ at $x_i^{1}$ under the metric $\hat{g}$, and thus we can write $P=\{x_{n+1}=0\}$ in the coordinates introduced before (\ref{Eq: boundary condition of FBMH on horizontal plane}). 
    Again, 
    we have the graphical limit $({\rm Graph}(\hat{u}^1),\{\bd_m {\rm Graph}(\hat{u}^1)\})\subset(\Omega_r, \{\bd_{m+1} \Omega_r\})$ for some { $\hat{u}^1\in C^{\infty}(P\cap \Omega_r\setminus \bd^E\Omega)\cap C^{1,1}(P\cap \Omega_r)$.  
    Note the convergence is $C^{1,\alpha}$-uniformly for $\alpha\in(0,1)$, while the limit is $C^{1,1}$ by a standard argument.}

    Next, for any other $p_i^{2}\in\Sigma_i^2$ with $\lim_{i\to\infty}p_i^2 = q_0\in\spt(\|V\|)$. 
    The above arguments show that $\Sigma_{i}^2$ converges graphically to some almost properly embedded locally wedge-shaped FBMH $\Sigma^2={\rm Graph}(u^2)\subset \spt(\|V\|)$ in a neighborhood of $q_0$. 
    Suppose $q\in\Sigma^1\cap\Sigma^2\neq\emptyset$. It remains to show that $\Sigma^1=\Sigma^2$. 
    
    For the case that $q\notin \bd^E\Sigma^1 \cup \bd^E\Sigma^2$, 
    the standard transversal argument and the classical maximum principle for FBMHs indicate $\Sigma_1=\Sigma_2$. 
    
    If $q\in \bd^E\Sigma^1\cap (\Sigma^2\setminus\bd^E\Sigma^2)$ or $q\in \bd^E\Sigma^2\cap (\Sigma^1\setminus\bd^E\Sigma^1)$, it follows from the orthogonal meeting conditions Proposition \ref{Prop: classify FBMH}(2)(3) (see also Remark \ref{Rem: orthogonally meeting condition}) that $T_q\Sigma^1$ is transversal to $T_q\Sigma^2$, which contradicts the embeddedness of $\Sigma_i$ for $i$ large enough. 
    
    If $q\in \bd^E\Sigma^1\cap \bd^E\Sigma^2$, then by Remark \ref{Rem: orthogonally meeting condition} and a transversal argument, $\Omega':=T_q\Sigma^1=T_q\Sigma^2=P\cap T_pM$ for some horizontal hyperplane $P\in\Pcal_{T_pM}$, and $\Sigma^1$ lies on one side of $\Sigma^2$ in a neighborhood of $q$. 
    The maximum principle (Theorem \ref{Thm: maximum principle}) then suggests $\Sigma_1=\Sigma_2$. 
    This proved the theorem. 
\end{proof}

\section{Bernstein-type theorem in Euclidean wedge domains} 
\label{sec:Bernstein}

The goal of this section is to prove a Bernstein-type theorem for stable almost properly embedded locally wedge-shaped FMBHs.  We obtain the strongest statement when the wedge angle satisfies $\theta \le \pi/2$. For angles larger than $\pi/2$, the existence of a touching set complicates the application of the stability inequality. Nevertheless, for angles larger than $\pi/2$ we can still obtain a Bernstein theorem under the assumption that the touching set is empty or it satisfies a stronger stability assumption. 

To begin with, we study the touching phenomenon for FBMHs in Euclidean wedges with acute or right wedge angle. 
\begin{proposition}
\label{prop:acute-touching-set}
Let $\Omega^{n+1} = H_+^{n+1} \cap H_-^{n+1}$ be an $(n+1)$-dimensional wedge domain in $\R^{n+1}$ with wedge angle $\theta < \pi/2$.  Then any $C^{1,1}$-to-edge almost properly embedded FBMH $(\Sigma,\{\bd_k\Sigma\}) \subset (\Omega,\{\bd_{k+1}\Omega\})$ is properly embedded. 
\end{proposition}

\begin{proof}
    We need to show that the touching set $\mathcal S(\Sigma) = (\operatorname{int}(\Sigma)\cap \bd^F\Omega)\cup \bd^{FE}\Sigma$ is empty. According to Proposition \ref{Prop: classify FBMH}, the unit co-normal $\eta$ to $\bd^F \Sigma$ satisfies $\eta \perp \bd^+\Omega$ or $\eta\perp \bd^-\Omega$ at all points $p\in \bd^{FE}\Sigma$. Since the wedge angle $\theta$ is less than $\pi/2$, this implies that $\bd^{FE}\Sigma = \emptyset$. 

    Now suppose for contradiction there is a point $p\in \operatorname{int}(\Sigma)\cap \bd^F \Omega$. Without loss of generality, we can assume that $p\in \bd^+ \Omega$.  Since $\Sigma$ lies to one side of $\bd H_+$, the maximum principle implies that $\Sigma = \bd^+ \Omega$. This is not a free boundary minimal hypersurface according to Definition \ref{definition:FBMH} and $\theta<\pi/2$, and so this  is impossible. 
\end{proof}

\begin{proposition}
\label{prop:right-touching-set}
Let $\Omega^{n+1} = H_+^{n+1} \cap H_-^{n+1}$ be an $(n+1)$-dimensional wedge domain in $\R^{n+1}$ with wedge angle $\theta = \pi/2$.  Then any $C^{1,1}$-to-edge almost properly embedded FBMH $(\Sigma,\{\bd_k\Sigma\}) \subset (\Omega,\{\bd_{k+1}\Omega\})$ with non-empty touching set is flat. 
\end{proposition}

\begin{proof}
    As in the previous proposition, the maximum principle implies that if $\operatorname{int}(\Sigma) \cap \bd^F\Omega$ is non-empty then $\Sigma$ coincides with either $\bd^+ \Omega$ or $\bd^- \Omega$. 

    Now suppose there is a point $p\in \bd^{FE}\Sigma$. Then according to Proposition \ref{Prop: classify FBMH}, the unit co-normal $\eta$ to $\bd^F \Sigma$ satisfies $\eta(p) \perp \bd^+\Omega$ or $\eta(p)\perp \bd^-\Omega$. Without loss of generality, suppose that $\eta(p)\perp \bd^+\Omega$. 
    {Then by Remark \ref{Rem: orthogonally meeting condition} and $\theta=\pi/2$, we see $\Sigma$ and $\bd^-\Omega$ are two FBMHs in a small neighborhood of $p$ in $H_+$ with no edge point, and $\Sigma$ meets $\bd^-\Omega$ tangentially at $p$ from one side. 
    Hence, $\Sigma=\bd^-\Omega$ is flat by the maximum principle for FBMHs.} 
\end{proof}

\begin{theorem}\label{Thm: Bernstein theorem}
    Let $3\leq n+1\leq 6$, and $\Omega^{n+1}$ be an $(n+1)$-dimensional wedge domain in $\R^{n+1}$ with wedge angle $\theta \le \pi/2$. Assume that $(\Sigma,\{\bd_k\Sigma\})$ is a $C^{1,1}$-to-edge almost properly embedded FBMH in $(\Omega,\{\bd_{k+1}\Omega\})$ which is stable according to Definition \ref{definition:stability}. Further assume that $\Sigma$ has Euclidean area growth: there is a constant $C > 0$ such that 
    \[
    \mathcal H^n(\Sigma\cap \mathbb {B}^{n+1}_r(0))\le Cr^n
    \]
    for all $r > 0$. Then $\Sigma = \Omega \cap P$ where $P\subset \R^{n+1}$ is an $n$-plane. If $\theta < \pi/2$, then $P$ is {a horizontal hyperplane of $\Omega$ (\ref{Eq: horizontal plane in domain}) after a translation along $\bd^E\Omega$.}
\end{theorem}

\begin{proof}
By Propositions \ref{prop:acute-touching-set} and \ref{prop:right-touching-set}, we can assume that $\Sigma$ is properly embedded. Let $\nu$ be the normal vector to $\Sigma$. Since $\Sigma$ is properly embedded, for any function $f$ on $\Sigma$, the vector field $f\nu$ on $\Sigma$ extends to an ambient vector field $X\in \mathfrak X(\Omega,\Sigma)$. In particular, the stability of $\Sigma$ implies that 
\begin{equation}
\label{equation:strong-stability}
Q(f,f) = \int_{\Sigma} \vert \grad f\vert^2 - \vert A\vert^2 f^2 + \int_{\bd \Sigma} \langle \del_{f\nu} (f\nu), \eta\rangle \ge 0
\end{equation}
for all $f$ which are restrictions of smooth ambient functions to $\Sigma$.  Next, observe that $\mathcal  H^{n-1}(\bd^E \Sigma) = 0$ by (\ref{Eq: edge is closed manifold}). Moreover, at any point $p\in \bd^F \Sigma$ we have 
\[
\lan \del_{f\nu}(f\nu), \eta\ran = f^2  A_{\bd^{\pm}\Omega}(\nu,\nu) = 0
\]
since $\bd^{+}\Omega$ and $\bd^{-}\Omega$ are both flat.  Therefore, the boundary integral vanishes and we in fact have 
\[
\int_{\Sigma} \vert \grad f\vert^2 - (\vert A\vert^2 - \ric^M(\nu,\nu))f^2 \ge 0
\]
for all $f$. 

We now proceed to use the Schoen-Simon-Yau estimates \cite{schoen1975curvature} as in the proof of the usual Bernstein theorem. Here we largely follow \cite[Theorem 2.21] {colding2011course}. However, one needs to be slightly careful, as the $C^{1,1}$-to-edge regularity does not a priori ensure that $\Delta \vert A\vert$ is integrable all the way to the edge. Thus we shall make use of the log-cutoff trick. 

Specifically, given a compactly supported test function $f$, we can take $R>0$ so that $\spt(f)\subset\Sigma\cap \mB^{n+1}_R(0)$. Then, define cutoff functions $\eta_r$ on $\Sigma$ which vanish in a neighborhood of $\bd^E \Sigma$ and which converge pointwise to 1 on $\Sigma\setminus \bd^E\Sigma$ as $r\to 0$. Since $\bd^E\Sigma$ is of codimension 2 in $\Sigma$, we can arrange that
\[
\int_{\Sigma\cap \mB^{n+1}_{2R}(0)} \vert \grad \eta_r\vert^2 \to 0
\]
as $r\to 0$ (cf. \cite[(1.103)]{colding2011course}). 
Now 
define $g_r := f \eta_r$. Apply the stability inequality with the test function $\vert A\vert^{1+q} g_r$ to get
\begin{align}
\label{eqn:bern1}
\int_\Sigma \vert A\vert^{4+2q}g_r^2 &\le (1+q)^2 \int_\Sigma g_r^2 \vert \grad \vert A\vert \vert^2 \vert A\vert^{2q} + \int_{\Sigma} \vert A\vert^{2+2q}\vert \grad g_r\vert^2 \\
&\qquad + 2(1+q)\int_\Sigma g_r \vert A\vert^{1+2q}\laa \grad g_r,\grad \vert A\vert \raa.\nonumber
\end{align}
Next, multiply Simons' inequality by $\vert A\vert^{2q} g_r^2$ and then integrate. 
This yields 
\begin{align}
\label{eqn:bern2}
    \frac{2}{n}\int_\Sigma \vert \grad \vert A\vert \vert^2 \vert A\vert^{2q}g_r^2 &\le \int_{\Sigma} \vert A\vert^{4+2q}g_r + \int_\Sigma g_r^2 \vert A\vert^{1+2q} \Delta \vert A\vert \nonumber \\
    &= \int_\Sigma \vert A\vert^{4+2q}g_r^2 - 2\int_{\Sigma} g_r \vert A\vert^{1+2q} \lan \grad g_r,\grad \vert A\vert \ran \\
    &\qquad - (1+2q) \int_{\Sigma} g_r^2 \vert A\vert^{2q} \vert \grad \vert A\vert\vert^2 + \int_{\bd \Sigma} g_r^2 \vert A\vert^{1+2q} \lan \grad \vert A\vert,\eta\ran. \nonumber
\end{align}
Note there is an extra boundary term here compared with the usual calculation \cite[Equation (2.100)]{colding2011course}. 

We claim that this extra boundary term vanishes. 
Namely, we want to show that 
\[
\int_{\bd \Sigma} g_r^2 \vert A\vert^{1+2q}\lan  \grad \vert A\vert,\eta\ran = 0. 
\]
Since {$\mathcal H^{n-1}(\bd^E\Sigma)=0$}, it suffices to show that $\lan \grad \vert A\vert,\eta\ran(p) = 0$ for all $p\in \bd^F \Sigma$. Since $\Sigma$ is properly embedded, we can assume without loss of generality that $p\in \bd^+\Omega$. Let $\tilde \Sigma$ be the hypersurface consisting of $\Sigma$ together with the reflection of $\Sigma$ across $\bd H_+$. Then $\tilde \Sigma$ is a smooth minimal hypersurface with no boundary in a neighborhood of $p$ in $\R^{n+1}$. Now let $\gamma\colon (-1,1) \to \tilde \Sigma$ be a curve with 
\[\gamma(0) = p,\qquad \gamma'(0) = \eta,\qquad {\rm and}\quad \gamma(t) = R(\gamma(-t)),\]
where $R$ denotes reflection across $\bd H_+$. Then $\vert A\vert(\gamma(t))$ is an even function of $t$ and so 
\[
0 = \frac{d}{dt}\bigg\vert_{t=0} \vert A\vert(\gamma(t)) = \lan \grad \vert A\vert,\eta\ran(p). 
\]
Since $p\in \bd^F \Sigma$ was arbitrary, the boundary integral therefore vanishes as claimed.

Equations (\ref{eqn:bern1}) and (\ref{eqn:bern2}) can now be combined exactly as in \cite{colding2011course} to show that 
\begin{equation}
    \label{eqn:bern3}
    \int_\Sigma \vert A\vert^{4+2q}g_r^2 \le \left(\frac{2(1+q)^2(1+\frac{q}{\eps})}{\frac{2}{n}-q^2-\eps q} + 2\right)\int_\Sigma \vert A\vert^{2+2q}\vert \grad g_r\vert^2
\end{equation}
for any sufficiently small $\eps > 0$. Since the energy of the logarithmic cutoff function $\eta_r$ in $\mB^{n+1}_{2R}(0)$ goes to 0 as $r\to 0$ and $\spt(f)\subset 
\Sigma\cap \mB^{n+1}_R(0)$, we can send $r\to 0$ in the previous equation to obtain 
\begin{equation}
    \label{eqn:bern4}
    \int_\Sigma \vert A\vert^{4+2q}f^2 \le \left(\frac{2(1+q)^2(1+\frac{q}{\eps})}{\frac{2}{n}-q^2-\eps q} + 2\right)\int_\Sigma \vert A\vert^{2+2q}\vert \grad f\vert^2.
\end{equation}
Given (\ref{eqn:bern4}), the remainder of the proof can now be completely exactly as in \cite{colding2011course}.
\end{proof}

\begin{remark}\label{Rem: Bernstein theorem}
When $\theta > \pi/2$ the same result still holds for properly embedded FBMHs {and almost properly embedded FBMHs satisfying (\ref{equation:strong-stability})}. However, for almost properly embedded FBMHs it is not clear in general that stability in the sense of Definition \ref{definition:stability} implies (\ref{equation:strong-stability}).
\end{remark}

\section{Curvature estimates for stable free boundary minimal hypersurfaces}\label{Sec: curvature estimates}

In this section, we show the main curvature estimates for stable almost properly embedded locally wedge-shaped FBMHs in any compact locally wedge-shaped Riemannian manifold $M\subset \R^L$. 
Since the curvature estimates are well known in $\interior (M)\cup \bd^F M$ (c.f. \cite{schoen1975curvature}\cite{schoen1981regularity}\cite{guang2020curvature}), we mainly focus on $p\in \bd^EM$. 

The following lemma is a simple observation, which is helpful to show the Euclidean area growth in the proof of our main theorem. 
\begin{lemma}\label{Lem: dist to edge bounded by dist to face}
    Let $M^{n+1}\subset \R^L$ be a locally wedge-shaped Riemannian manifold (with induced metric). 
    Then for any $p\in\bd^EM$, there exists $r_0=r_0(p)>0$ so that for any $q\in M\cap \mB^L_{r_0}(p)$, 
    $$ \frac{1}{2}\sin(\theta(p))\dist_{\R^L}(q,\bd^E M) \leq \dist_{\R^L}(q,\bd^+ M_{p,r_0}) + \dist_{\R^L}(q, \bd^- M_{p,r_0}) $$
    where $\bd^{\pm} M_{p,r_0}$ are defined as in (\ref{Eq: local extension}).  
\end{lemma}
\begin{proof}
    Suppose the assertion is not true. 
    Then we can find a sequence $\{q_i\}_{i\in\N}$ in $M$ so that $q_i\to p\in\bd^EM$ as $i\to\infty$, and $ \frac{1}{2} \sin(\theta(p))d_i> d_i^+ +d_i^->0$, where $d_i:=\dist_{\R^L}(q_i,\bd^E M)$ and $d_i^{\pm}:=\dist_{\R^L}(q_i,\bd^{\pm} M_{p,r_0})$. 
    Denote by $x_i$ and $y_i^{\pm}$ the nearest projection (in $\R^L$) of $q_i$ to $\bd^EM$ and $\bd^{\pm} M_{p,r_0}$ respectively. 
    Then consider the blow-up (conformal) maps $\eta_i(x):=(x-q_i)/d_i$, and take $M_i:=\eta_i(M)$. 
    By the smoothness of $M$, we see $M_i$ converges to a wedge domain $\Omega=H_+ \cap H_-$ containing $0$ with wedge angle $\theta(p)$ in the $(n+1)$-subspace containing $T_pM$, where $H_{\pm}$ are two affine closed half $(n+1)$-spaces.  
    In addition, we have
    \begin{itemize}
        \item $\eta_i(x_i)\to x \in \bd^E\Omega\cap\bd\mB^L_1(0)$, 
        ~$\eta_i(y_i^{\pm})\to y^{\pm} \in \bd^{\pm}\Omega \cap \Clos(\mB^L_1(0))$, 
        \item $\frac{1}{2}\sin(\theta(p)) \geq |y^+| + |y^-|$,
    \end{itemize}
    where $x$ and $y^{\pm}$ are the nearest point of $0$ to $\bd^E\Omega$ and $\bd^{\pm}\Omega$ (in $\R^L$) respectively.
    However, in the Euclidean wedge domain $\Omega$, we have $|y^+| + |y^-| \geq \dist_{\R_L}(0,\bd H_+) + \dist_{\R^L}(0, \bd H_-) = \sin(\theta_+)+\sin(\theta_-) \geq \sin(\theta_+)\cos(\theta_-)+\sin(\theta_-)\cos(\theta_+) = \sin(\theta_++\theta_-)=\sin(\theta(p))$, where $\theta_{\pm}\in [0,\theta(p)]$ are the wedge angle between $\bd^{\pm}\Omega$ and the half $n$-plane containing $\bd^E\Omega\cup \{0\}$.
    This contradicts the second bullet. 
\end{proof}

For simplicity, we use the terminologies $M_{p,r}:=\mB^L_r(p)\cap M$, $ \bd_m M_{p,r} := \mB^L_r(p)\cap \bd_m M $, and $\bd^EM_{p,r}:=\mB^L_r(p)\cap \bd^EM$ as we did after (\ref{Eq: local extension}), where $p\in\bd^EM$ and $r>0$. 
Additionally, we denote by $r_0>0$ the radius (depending on $M$) so that the monotonicity formula (Proposition \ref{Prop: monotonicity fomula}) and Lemma \ref{Lem: small curvature implies graphical} hold for any $r\in (0,r_0)$ and $p\in\bd^E M$.

\begin{theorem}\label{Thm: curvature estimates}
    Let $2\leq n\leq 5$, $M^{n+1}\subset \R^L$ be a locally wedge-shaped Riemannian manifold (with induced metric) satisfying (\ref{dag}) or (\ref{ddag}), and $r_0$ be given as above. 
    Suppose $p\in\bd^EM$, $r\in (0,r_0)$, and  $(\Sigma, \{\bd_m\Sigma\}) \subset ( M \cap \mB^L_r(p), \{\bd_{m+1}M \cap \mB^L_r(p)\})$ is an almost properly embedded, 
    $C^{1,1}$-to-edge locally wedge-shaped, stable FBMH so that ${\rm Area}(\Sigma\cap \mB^L_r(p)) \leq C_1$. 
    Then 
    $$ \sup_{x\in \Sigma\cap\mB_{r/2}^L(p)} |A_\Sigma|(x) \leq C_2, $$
    where $C_2=C_2(C_1,M)>0$. 
\end{theorem}
\begin{remark}
    By Proposition \ref{prop:acute-touching-set} and \ref{prop:right-touching-set}, under the wedge angle assumption $\theta(q)\leq \pi/2$ for all $q\in \mB^L_{r_0}(p)\cap\bd^EM$, the global stability of $\Sigma$ can be weakened to the {\em stability away from the touching set} (naturally generalized from \cite[Definition 2.2]{guang2021compactness}) to obtain the Bernstein-type theorem \ref{Thm: Bernstein theorem} and also the above curvature estimates. 
\end{remark}
\begin{proof}
    The theorem is obtained by a contradiction argument. 
    Specifically, suppose there exists a sequence $(\Sigma_i, \{\bd_m\Sigma_i\}) \subset ( M_{p,r}, \{\bd_{m+1}M_{p,r}\})$ of properly embedded, locally wedge-shaped, stable FBMHs so that
    \begin{equation}\label{eq: proof curvature estimates 1}
        {\rm Area}(\Sigma_i\cap \mB^L_r(p)) \leq C, \qquad {\rm but}\qquad \sup_{x\in\Sigma_i\cap \mB^L_{r/2}(p)} |A_{\Sigma_i}|(x) \to \infty,~{\rm as}~i\to\infty.
    \end{equation}
    Hence, for each $i\in\N$, there exists $x_i\in \Sigma_i\cap \mB^L_{r/2}(p)$ so that $\lim_{i\to\infty}|A_{\Sigma_i}|(x_i) = \infty$ and $x_i\to x\in \mB^L_{3r/4}(p)$. 
    Note we must have $x\in \bd^EM$ due to the curvature estimates in the interior (Schoen–Simon–Yau \cite{schoen1975curvature}) and on the smooth boundary (Guang-Li-Zhou \cite{guang2020curvature}). 
    Define 
    $$ r_i := \left( |A_{\Sigma_i}|(x_i) \right)^{-\frac{1}{2}} \to 0\qquad {\rm as}~ i\to\infty,$$
    and thus $r_i |A_{\Sigma_i}(x_i)| \to\infty$ as $i\to\infty$. 
    
    Since each $\Sigma_i$ is $C^{1,1}$-to-edge, we can take $y_i\in \Sigma_i\cap\mB^L_{r_i}(x_i)$ so that:
    \begin{equation}\label{eq: proof curvature estimates 2}
        |A_{\Sigma_i}(y_i)|\cdot \dist_{\R^L}(y_i,\bd \mB_{r_i}^L(x_i)) \geq \frac{i}{i+1}\sup_{y\in \Sigma_i\cap\mB^L_{r_i}(x_i)} |A_{\Sigma_i}(y)|\cdot \dist_{\R^L}(y,\bd \mB_{r_i}^L(x_i)).
    \end{equation}
    Define $r_i':= r_i - |y_i-x_i|$, which tends to $0$ since $r_i'\in (0,r_i)$. 
    Noting $\mB^L_{r_i'}(y_i)\subset \mB^L_{r_i}(x_i)$, one easily verifies that 
    \begin{equation}\label{eq: proof curvature estimates 3}
        |A_{\Sigma_i}(y_i)|\cdot \dist_{\R^L}(y_i,\bd \mB_{r_i'}^L(y_i)) \geq \frac{i}{i+1}\sup_{y\in \Sigma_i\cap\mB^L_{r_i'}(y_i)} |A_{\Sigma_i}(y)|\cdot \dist_{\R^L}(y,\bd \mB_{r_i'}^L(y_i)).
    \end{equation}
    Consider $\lambda_i := |A_{\Sigma_i}|(y_i)$. 
    Then, by $\dist_{\R^L}(y_i,\bd \mB_{r_i'}^L(y_i))= \dist_{\R^L}(y_i,\bd \mB_{r_i}^L(x_i))$ and (\ref{eq: proof curvature estimates 2}),
    \begin{eqnarray}\label{eq: proof curvature estimates 4}
        \lambda_i r_i' &=& |A_{\Sigma_i}|(y_i)\dist_{\R^L}(y_i,\bd \mB_{r_i'}^L(y_i)) \\
        &\geq& \frac{i}{i+1}|A_{\Sigma_i}|(x_i)\dist_{\R^L}(x_i,\bd \mB_{r_i}^L(x_i)) \geq \frac{1}{2} r_i |A_{\Sigma_i}|(x_i) \to \infty, \nonumber
    \end{eqnarray}
    and thus $\lambda_i\to \infty$ as $i\to\infty$. 

    Now we take the various blow-up maps $\eta_i (z) := \lambda_i (z-y_i)$, and consider the blow-up sequence of almost properly embedded, locally wedge-shaped, stable FBMHs 
    $$ (\Sigma_i', \{\bd_m\Sigma_i'\}) ~\subset~ (M_i\cap \mB^L_{\lambda_ir}(0), \{\bd_{m+1} M_i \cap \mB^L_{\lambda_ir}(0)\}) ,$$
    where $\Sigma_i':= \eta_i(\Sigma_i)$ and $M_i := \eta_i(M)$. 
    Note 
    \begin{equation}\label{eq: proof curvature estimates 5.0}
        |A_{\Sigma_i'}|(0) = 1.
    \end{equation}
    Moreover, for any fixed $r'>0$, we claim the following inequality for $i$ large enough:
    \begin{equation}\label{eq: proof curvature estimates 5}
        |A_{\Sigma_i'}|(z) \leq (1+\frac{1}{i})\frac{\lambda_i r_i'}{\lambda_ir_i' - r'} \to 1, \qquad\forall z\in\Sigma_i'\cap\mB^L_{r'}(0).
    \end{equation}
    Indeed, by (\ref{eq: proof curvature estimates 4}), we have $r'/\lambda_i < r_i'$ for $i$ large enough. 
    Then, for any $z\in \Sigma_i'\cap\mB^L_{r'}(0)$, $\eta_i^{-1}(z)\in\Sigma_i\cap\mB^L_{r'/\lambda_i}(y_i)\subset\Sigma_i\cap\mB^L_{r_i'}(y_i)$ and 
    $$\dist_{\R^L}(\eta_i^{-1}(z), \bd \mB^L_{r_i'}(y_i)) = r_i' - |\eta_i^{-1}(z) - y_i| \geq r_i' - r'/\lambda_i. $$
    Hence, $\frac{i}{i+1}|A_{\Sigma_i}|(\eta_i^{-1}(z)) \cdot  (r_i'-r'/\lambda_i) \leq \frac{i}{i+1}|A_{\Sigma_i}|(\eta_i^{-1}(z)) \cdot \dist_{\R^L}(\eta_i^{-1}(z), \bd \mB^L_{r_i'}(y_i)) \leq \lambda_i r_i'$ by (\ref{eq: proof curvature estimates 3}), which indicates (\ref{eq: proof curvature estimates 5}). 

    Noting $\lim_{i\to\infty} y_i = x\in\bd^EM$, the smoothness of $M$ indicates that $M_i$ converges to a domain in the $(n+1)$-subspace containing $T_xM$ smoothly and locally uniformly in $\R^L$. 
    Then the arguments are separated into three cases:
    \begin{itemize}
        \item[(I)] $\liminf_{i\to\infty} \lambda_i\cdot \dist_{\R^L}(y_i,\bd M) = \infty$;
        \item[(II)]  $\liminf_{i\to\infty} \lambda_i\cdot \dist_{\R^L}(y_i,\bd M) < \infty$, and $\liminf_{i\to\infty} \lambda_i\cdot \dist_{\R^L}(y_i,\bd^E M) = \infty$;
        \item[(III)] $\liminf_{i\to\infty} \lambda_i\cdot \dist_{\R^L}(y_i,\bd^E M) < \infty$. 
    \end{itemize}
    The basic idea is to prove in each case that $\Sigma_i'$ satisfies a uniformly Euclidean area growth and a curvature bound (\ref{eq: proof curvature estimates 5}), and thus the convergence theorem implies a subsequence of $\{\Sigma_i'\}$ tends to a stable minimal hypersurface $\Sigma_\infty$ with Euclidean area growth and non-vanishing second fundamental form at $0$, which contradicts the Bernstein-type theorem. 
    
    For simplicity, in the following arguments, let $C>0$ be a varying uniform constant, 
    $$d_i := \dist_{\R^L}(y_i, \bd^EM) \geq \dist_{\R^L}(y_i, \bd M) = \min\{d_i^+, d_i^-\}, \quad d_i^{\pm} := \dist_{\R^L}(y_i, \bd^{\pm} M_{p,r_0}),$$ and $z_i, z_i^{\pm}$ be the nearest projection (in $\R^L$) of $y_i$ to $\bd^EM$ and $\bd^{\pm} M_{p,r_0}$ respectively. 
    Without loss of generality, we always assume $\min\{d_i^+, d_i^-\} = d_i^+$. 
    
    {\bf Case (I).} In this case, the boundary $\bd M_i$ will disappear in the limit, i.e. 
    \[M_i\to \R^{n+1}=\mbox{the $(n+1)$-space containing $T_pM$}.\]
    For any fixed $r'>0$, noting $\lambda_id_i^+\to\infty$, we have $r'/\lambda_i<d_i^+ \to 0$ for $i$ large, and
    \begin{equation}\label{eq: proof curvature estimates inclusion 1}
        \mB^L_{r'/\lambda_i}(y_i) ~\subset~ \mB^L_{d_i^+}(y_i) ~\subset ~\mB^L_{2d_i^+}(z_i^+) \cap \mB^L_{d_i + d_i^+}(z_i) .
    \end{equation}
    Since $\mB^L_{d_i^+}(y_i)\cap \bd M=\emptyset$, the interior monotonicity formula \cite[17.6]{simon1983lectures} indicates
    \begin{equation}\label{eq: proof curvature estimates Euclidean area growth 0}
        \frac{{\rm Area}(\Sigma_i\cap\mB^L_{r'/\lambda_i}(y_i))}{(r'/\lambda_i)^n} \leq C\frac{{\rm Area}(\Sigma_i\cap\mB^L_{d_i^+}(y_i))}{(d_i^+)^n} .
    \end{equation}
    
    {\bf Sub-case (I.a):} $\liminf_{i\to\infty} d_i/d_i^+ = \infty$. 
    Then by Lemma \ref{Lem: dist to edge bounded by dist to face}, $\liminf_{i\to\infty}d_i^-/d_i^+=\infty$. 
    Let $R_i:=\dist_{\R^L}(z_i^+, \bd^-M_{p,r_0}) \to 0$, which satisfies $d_i^-\leq d_i^++R_i$. 
    Hence, $R_i > 2d_i^+ $ for $i$ large enough. 
    Noting $\mB^L_{R_i}(z_i^+)\cap \bd^-M_{p,r_0}=\emptyset$, we can use (\ref{eq: proof curvature estimates inclusion 1})(\ref{eq: proof curvature estimates Euclidean area growth 0}) and the boundary monotonicity formula \cite[Theorem 3.4]{guang2020curvature} to show
    $$ \frac{{\rm Area}(\Sigma_i\cap\mB^L_{r'/\lambda_i}(y_i))}{(r'/\lambda_i)^n} \leq 2^nC\frac{{\rm Area}(\Sigma_i\cap\mB^L_{2d_i^+}(z_i^+))}{(2d_i^+)^n} \leq 2^nC\frac{{\rm Area}(\Sigma_i\cap\mB^L_{R_i}(z_i^+))}{(R_i)^n}. $$
    Let $w_i\in\bd^EM$ be the nearest projection of $z_i^+$ to $\bd^EM$ (in $\R^L$). 
    By Lemma \ref{Lem: dist to edge bounded by dist to face}, $|z_i^+ - w_i| \leq 2 R_i/\sin(\theta(p))$ and 
    $$ \mB^L_{2d_i^+}(z_i^+) \subset \mB^L_{R_i}(z_i^+) \subset \mB^L_{(1+2/\sin(\theta(p)))R_i}(w_i) \subset \mB^L_{r/2}(w_i) \subset \mB^L_{r}(p).$$
    Thus, the edge monotonicity formula (Proposition \ref{Prop: monotonicity fomula}) and the above computations imply
    \begin{eqnarray}\label{eq: proof curvature estimates Euclidean area growth 1.1}
        \frac{{\rm Area}(\Sigma_i\cap\mB^L_{r'/\lambda_i}(y_i))}{(r'/\lambda_i)^n} &\leq& 2^n \frac{(\sin(\theta(p))+2)^n}{\sin^n(\theta(p))} C\frac{{\rm Area}(\Sigma_i\cap\mB^L_{(1+2/\sin(\theta(p)))R_i}(w_i))}{[(1+2/\sin(\theta(p)))R_i]^n} \nonumber\\
        &\leq& C \frac{{\rm Area}(\Sigma_i\cap\mB^L_{r/2}(w_i))}{(r/2)^n}\leq C \frac{{\rm Area}(\Sigma_i\cap\mB^L_{r}(p))}{r^n} \leq CC_1r^{-n}.
    \end{eqnarray}
    
    {\bf Sub-case (I.b):} $\liminf_{i\to\infty} d_i/d_i^+ < \infty$, i.e. $d_i/d_i^+$ is uniformly bounded. 
    Then by (\ref{eq: proof curvature estimates inclusion 1})(\ref{eq: proof curvature estimates Euclidean area growth 0}) and the edge monotonicity formula (Theorem \ref{Prop: monotonicity fomula}), we have 
    \begin{eqnarray}\label{eq: proof curvature estimates Euclidean area growth 1.2}
         \frac{{\rm Area}(\Sigma_i\cap\mB^L_{r'/\lambda_i}(y_i))}{(r'/\lambda_i)^n} &\leq& C(1+d_i/d_i^+)^n \frac{{\rm Area}(\Sigma_i\cap\mB^L_{d_i+d_i^+}(z_i))}{(d_i+d_i^+)^n} \nonumber \\
         &\leq & C \frac{{\rm Area}(\Sigma_i\cap\mB^L_{r/2}(z_i))}{(r/2)^n}\leq C \frac{{\rm Area}(\Sigma_i\cap\mB^L_{r}(p))}{r^n} \leq CC_1r^{-n}.
    \end{eqnarray}

    Therefore, by (\ref{eq: proof curvature estimates Euclidean area growth 1.1}) and (\ref{eq: proof curvature estimates Euclidean area growth 1.2}), $\Sigma_i'$ satisfies the uniform Euclidean area growth
    \begin{equation}\label{eq: proof curvature estimates Euclidean area growth}
        {\rm Area}(\Sigma_i'\cap \mB^L_{r'}(0))\leq C(r')^n 
    \end{equation}
    for any fixed $r'>0$ and $i$ large enough. 
    After applying the classical convergence theorem for minimal hypersurfaces with bounded area and curvature (\ref{eq: proof curvature estimates 5}), a subsequence of $\Sigma_i'$ converges smoothly and locally uniformly to a complete embedded stable minimal hypersurface $\Sigma_\infty$ in the Euclidean $(n+1)$-space containing $T_pM$, which also satisfies the above Euclidean area growth and $0\in \Sigma_\infty$. 
    It follows from the classical Bernstein theorem \cite[Theorem 2]{schoen1975curvature} that $\Sigma_\infty$ is an $n$-plane, which contradicts $|A_{\Sigma_\infty}|(0) = 1$ in (\ref{eq: proof curvature estimates 5.0}). 

    {\bf Case (II).} In this case, the edge $\bd^E M_i$ will disappear in the limit ($\liminf_{i\to\infty} \lambda_id_i=\infty$), while the face $\bd^+M_{p,r_0} $ will not ($\liminf_{i\to\infty} \lambda_id_i^+<\infty$). 
    By Lemma \ref{Lem: dist to edge bounded by dist to face}, we also have $\liminf_{i\to\infty} \lambda_id_i^-=\infty$, i.e. the face $\bd^-M_{p,r_0} $ will disappear in the limit too. 
    Hence, 
    \[ M_i\to P_+ = \mbox{an affine half space in the $(n+1)$-space containing } T_pM. \]
    Additionally, let $R_i:=\dist_{\R^L}(z_i^+, \bd^-M_{p,r_0}) $, which satisfies $d_i^-\leq d_i^++R_i$, and thus $\liminf_{i\to\infty} \lambda_i R_i=\infty$. 
    Given any fixed $r'>0$, we can take $i$ large enough so that
    \[ R_i > d_i^++r'/\lambda_i \quad{\rm and}\quad \mB_{d_i^++r'/\lambda_i}(z_i^+)\cap \bd^-M_{p,r_0} = \emptyset.\]
    Denote by $w_i\in\bd^EM$ the nearest projection of $z_i^+$ to $\bd^EM$ (in $\R^L$), which satisfies $|w_i-z_i^+|\leq 2R_i/\sin(\theta(p))$ for $i$ large enough by Lemma \ref{Lem: dist to edge bounded by dist to face}.
    Then for $i$-large, 
    $$  \mB^L_{r'/\lambda_i}(y_i)\subset \mB^L_{d_i^++r'/\lambda_i}(z_i^+) \subset \mB^L_{R_i}(z_i^+) \subset \mB^L_{(1+2/\sin(\theta(p)))R_i}(w_i) \subset \mB^L_{r/2}(w_i) \subset \mB^L_{r}(p). $$
    Combining the above inclusions and the boundary/edge monotonicity formula (\cite[Theorem 3.4]{guang2020curvature}, Proposition \ref{Prop: monotonicity fomula}), we have
    \begin{eqnarray*}
        \frac{{\rm Area}(\Sigma_i\cap\mB^L_{r'/\lambda_i}(y_i))}{(r'/\lambda_i)^n} &\leq& (\lambda_i d_i^+/r' + 1)^n \frac{{\rm Area}(\Sigma_i\cap\mB^L_{d_i^++r'/\lambda_i}(z_i^+))}{(d_i^++r'/\lambda_i)^n} \nonumber \\
         &\leq & C \frac{{\rm Area}(\Sigma_i\cap\mB^L_{R_i}(z_i^+))}{R_i^n}\\
         &\leq & C \frac{(\sin(\theta(p))+2)^n}{\sin^n(\theta(p))} \frac{{\rm Area}(\Sigma_i\cap\mB^L_{(1+2/\sin(\theta(p)))R_i}(w_i))}{[(1+2/\sin(\theta(p)))R_i]^n}\\
         &\leq& C \frac{{\rm Area}(\Sigma_i\cap\mB^L_{r/2}(w_i))}{(r/2)^n}\leq C \frac{{\rm Area}(\Sigma_i\cap\mB^L_{r}(p))}{r^n} \leq CC_1r^{-n},
    \end{eqnarray*}
    where we used the facts that $\lambda_i d_i^+$ is uniformly bounded and $R_i>d_i^++r'/\lambda_i$ in the second inequality. 
    This indicates that $\Sigma_i'$ satisfies the uniform Euclidean area growth as in (\ref{eq: proof curvature estimates Euclidean area growth}) for any fixed $r'>0$ and $i$ large enough. 

    If for any $r'>0$ there is a subsequence $\Sigma_i$ so that its connected component passing through $y_i$ has no free boundary on $\bd^+M_{p,r_0}\cap\mB^L_{r'/\lambda_i}(y_i)$. 
    Then as in {\bf Case (i)}, the classical convergence theorem indicates that $\Sigma_i'$ converges smoothly and locally uniformly to a complete embedded stable minimal hypersurface $\Sigma_\infty$ in $P_+\subset \R^{n+1}$ satisfying the Euclidean area growth (\ref{eq: proof curvature estimates Euclidean area growth}), which can not touch $\bd P_+$ by the maximum principle and $|A_{\Sigma_\infty}|(0)=1$. Nevertheless, as a complete stable minimal hypersurface in $\R^{n+1}$ (since $\Sigma\cap\bd P_+=\emptyset$), the Bernstein theorem (\cite{schoen1975curvature}) also gives a contradiction as $|A_{\Sigma_\infty}|(0)=1$. 

    If there is $r'>0$ so that for any $i$ large enough, the connected component of $\Sigma_i$ passing through $y_i$ has free boundary on $\bd^+M_{p,r_0}\cap\mB^L_{r'/\lambda_i}(y_i)$. 
    By the convergence theorem (c.f. \cite[Theorem 6.1]{guang2020curvature}) for FBMHs with bounded area and curvature (\ref{eq: proof curvature estimates 5}), a subsequence of $\Sigma_i'$ converges to an embedded stable minimal hypersurface $(\Sigma_\infty, \bd\Sigma_\infty)\subset(P_+, \bd P_+ )$, which also satisfies the Euclidean area growth (\ref{eq: proof curvature estimates Euclidean area growth}) and $\bd\Sigma_\infty\neq\emptyset$. 
    Note also that $\interior(\Sigma_\infty)$ can not touch $\bd P_+$ due to the maximum principle and $|A_{\Sigma_\infty}|(0)=1$, which indicates $\Sigma_\infty$ is properly embedded. 
    Applying the reflection technique (\cite[Lemma 2.6]{guang2020curvature}) to $\Sigma_\infty$ across $\bd P_+$, we obtain a complete embedded stable minimal hypersurfaces in a Euclidean $(n+1)$-space with non-zero second fundamental form (\ref{eq: proof curvature estimates 5.0}) and Euclidean area growth, which contradicts the Bernstein theorem \cite[Theorem 2]{schoen1975curvature}. 


    {\bf Case (III).} In this case, we have 
    \[ M_i\to \Omega_{\theta(p)}^{n+1} = H_+\cap H_- = \mbox{a wedge domain in the $(n+1)$-space containing } T_pM, \]
    where $H_{\pm}$ are two affine half space. 
    For any fixed $r'>0$ and $i$ large enough, we see 
    $$ \mB^L_{r'/\lambda_i}(y_i) \subset \mB^L_{d_i+r'/\lambda_i}(z_i)\subset \mB^L_{r/2}(z_i)\subset \mB_{r}(p).$$ 
    As before, since $\{\lambda_id_i\}$ are uniformly bounded, the above inclusions and the edge monotonicity formula (Proposition \ref{Prop: monotonicity fomula}) shows 
    $$ \frac{{\rm Area}(\Sigma_i\cap\mB^L_{r'/\lambda_i}(y_i))}{(r'/\lambda_i)^n} \leq C_1r^{-n}2^nC(1+\lambda_id_i/r')^n \leq C. $$
    Therefore, $\Sigma_i'$ also satisfies the uniform Euclidean area growth as in (\ref{eq: proof curvature estimates Euclidean area growth}). 
    Applying the Convergence Theorem \ref{Thm: curvature bound implies convergence}, a subsequence $\{\Sigma_i'\}$ converges to an almost properly embedded locally wedge-shaped stable FBMH $(\Sigma_\infty, \{\bd_m\Sigma_\infty\})\subset (\Omega,\{\bd_{m+1}\Omega\})$, which also satisfies the uniform Euclidean area growth (\ref{eq: proof curvature estimates Euclidean area growth}). 
    Hence, the Bernstein-type theorem \ref{Thm: Bernstein theorem} indicates $\Sigma_\infty=P\cap \Omega$ is flat, where $P$ is an $n$-plane. 
    If $0\notin\bd^E \Sigma_\infty$, then the convergence near $0$ is smooth, so the contradiction with (\ref{eq: proof curvature estimates 5.0}) is achieved. 
    If $0\in\bd^E \Sigma_\infty$, then $\Sigma_i'\cap\mB_{r'}^L(0)$ can be written as a graph of some function $u_i$ on the Euclidean $n$-wedge $\Sigma_\infty\cap\mB_{r'}^L(0) = P\cap \Omega\cap\mB_{r'}^L(0)$ with $\|u_i\|_{C^{1,\alpha}(P\cap\Omega\cap\mB_{r'}^L(0))} \to 0$. 
    Using (\ref{dag})(\ref{ddag}) and the elliptic estimates in \cite[Theorem 2.17, Appendix C]{mazurowski2023min}, one can obtain that $\|u_i\|_{C^{2}(P\cap\Omega\cap\mB_{r'}^L(0))} \to 0$, which again contradicts (\ref{eq: proof curvature estimates 5.0}). 
\end{proof}

\appendix
	\addcontentsline{toc}{section}{Appendices}
	\renewcommand{\thesection}{\Alph{section}}
\section{Proof of the monotonicity formula}\label{Sec: monotonicity formula}

\begin{proof}[Proof of Proposition \ref{Prop: monotonicity fomula}.]
    By Lemma \ref{Lem: local chart}, there exist $r_0>0$ and $C_0>0$ depending only on $M\subset \R^L$ so that for any $p_0\in\bd^E M$ we have a local model $(\phi_{p_0}, \mB_{r_0}^L(p_0), \Omega_{p_0})$ of $M$ around $p_0$ satisfying
    \begin{equation}\label{Eq: mono formula (local chart)}
        \frac{1}{1+\epsilon} \leq |D\phi|_{C^0(\mB^L_{r_0}(0))} \leq 1+\epsilon, \quad |D^2\phi|_{C^0(\mB^L_{r_0}(0))} \leq C_0, \quad |D^2\phi^{-1}|_{C^0(\mB^L_{r_0}(0))} \leq C_0,
    \end{equation}
    where $\epsilon>0$ with $(1+\epsilon)^2\leq 2$. For simplicity, we write $\phi=\phi_{p_0}$, $\Omega=\Omega_{p_0}$. 
    
    Consider the position vector field $X(x) := x$ in $\R^L$, which satisfies $X\llcorner(\Omega\times 0^{L-n-1})\in \mathfrak{X}_{tan}(\Omega\times 0^{L-n-1})$. 
    Define then
    \[Y(p) := D\phi(X(\phi^{-1}(p)))\in T_p\R^L, \quad\mbox{for}~ p\in\mB^L_{r_0}(p_0).\]
    By a direct computation and (\ref{Eq: mono formula (local chart)}), we have
    \begin{equation}\label{Eq: mono formula (test vector field)}
        |Y(p)-(p-p_0)| \leq (1+\epsilon)C_0|p-p_0|^2 \quad\mbox{and}\quad |(DY)_p-\mbox{Id}| \leq (1+\epsilon)^2C_0|p-p_0|, 
    \end{equation}
    for $p\in\mB^L_{r_0}(p_0)$. 
    Without loss of generality, assume $p_0=0$ and write $Y(p) = p+(Y(p)-p)$. 
    
    Let $r(p):=|p|$ for $p\in\R^L$, and $\varphi:\R\to\R$ be a smooth function so that $\varphi\geq 0$, $\varphi'\leq 0$, and $\varphi\llcorner [r_0/2,\infty)=0$. 
    Then, we have $(\varphi\circ r)\cdot Y\llcorner M \in \mathfrak{X}_{tan}(M)$. 
    In addition, by (\ref{Eq: mono formula (test vector field)}) and the computation in \cite[Theore 3.4, Page 9]{guang2020curvature},  
    \begin{eqnarray*}
        \Div_S((\varphi\circ r)\cdot Y) &\geq & \varphi(r)\left(n-n(1+\epsilon)^2C_0 r\right) + \varphi'(r)\left[r(1-|\nabla^\perp_S r|^2) + (1+\epsilon)C_0r^2  \right]. 
    \end{eqnarray*}
    where $S\subset \R^L$ is an $n$-dimensional subspace. 
    Note $0=\delta V((\varphi\circ r)\cdot Y )$ since $V\in \mathcal{V}_n(M)$ is stationary in $M$ with free boundary. 
    Hence, combining the above inequality with the first variation formula (\ref{Eq: 1st variation for varifolds}), we conclude
    \[ \int n\varphi(r) + \int \varphi'(r) r\cdot\big(1+(1+\epsilon)C_0 r\big) \leq n(1+\epsilon)^2 C_0 \int\varphi(r) r + \int \varphi'(r) r\cdot |\nabla^\perp_S r|^2.  \]
    Now, for any $\rho\in (0, r_0/2)$, we take $\varphi(r) = \xi(\frac{r}{\rho})$, where $\xi : [0,\infty)\to [0,1]$ is a smooth cut-off function with $\xi'\leq 0$ and $\xi\llcorner[1,\infty)=0$. 
    Therefore, $\xi(\frac{r}{\rho})=0$ for $r\geq \rho$, and $r\varphi'(r) = -\rho\frac{d}{d\rho}(\xi(\frac{r}{\rho}))$. 
    After plugging these into the above inequality and adding $n(1+\epsilon)C_0\rho \int \xi(\frac{r}{\rho})$ on both sides, we see
    \begin{eqnarray*}
        n\big(1+(1+\epsilon)C_0\rho\big)\int\xi\left(\frac{r}{\rho}\right) - \rho \big(1+(1+\epsilon)C_0\rho\big)\frac{d}{d\rho}\int\xi\left(\frac{r}{\rho}\right) 
        \\
        \leq n(1+\epsilon)(2+\epsilon)C_0\rho\int\xi\left(\frac{r}{\rho}\right) - \rho\frac{d}{d\rho}\int\xi\left(\frac{r}{\rho}\right)|\nabla^\perp_S r|^2 .
    \end{eqnarray*}
    Denote by $I(\rho):= \int \xi(\frac{r}{\rho})$ and by $J(\rho):=\int\xi(\frac{r}{\rho})|\nabla^\perp_S r|^2$. 
    Then we have
    \[ \big(1+(1+\epsilon)C_0\rho\big)\frac{d}{d\rho} \left(\frac{I(\rho)}{\rho^n}\right) \geq \frac{J'(\rho)}{\rho^n} - n(1+\epsilon)(2+\epsilon)C_0\frac{I(\rho)}{\rho^n},\]
    and thus
    \[ \frac{d}{d\rho} \left(\frac{I(\rho)}{\rho^n}\right) + n(1+\epsilon)(2+\epsilon)C_0\frac{I(\rho)}{\rho^n} \geq \frac{J'(\rho)}{\big(1+(1+\epsilon)C_0\rho\big)\rho^n}, \]
    which implies 
    \[ \frac{d}{d\rho} \left[ e^{n(1+\epsilon)(2+\epsilon)C_0\cdot \rho}\cdot \frac{I(\rho)}{\rho^n}\right] \geq \int \frac{e^{n(1+\epsilon)(2+\epsilon)C_0\cdot \rho}}{\big(1+(1+\epsilon)C_0\rho\big)\rho^n} \cdot \frac{d}{d\rho}\left(\xi\left(\frac{r}{\rho}\right)\right) \cdot |\nabla^\perp_S r|^2 \geq 0.\]
    Next, suppose $\xi'\llcorner [0,(1-\delta)\rho]=0$ for some $\delta>0$ small enough. 
    Then, 
    \[ \frac{d}{d\rho} \left[ e^{n(1+\epsilon)(2+\epsilon)C_0\cdot \rho}\cdot \frac{I(\rho)}{\rho^n}\right] \geq \frac{d}{d\rho} \int \frac{e^{n(1+\epsilon)(2+\epsilon)C_0\cdot r}}{\left(1+(1+\epsilon)C_0\frac{r}{1-\delta}\right)\left(\frac{r}{1-\delta}\right)^n} \cdot \xi\left(\frac{r}{\rho}\right) \cdot |\nabla^\perp_S r|^2 .\]
    Finally, since $\epsilon, r_0, C_0>0$ are uniform constants, the monotonicity formula is obtained by taking the integral over $\rho\in [s,t]$ and letting $\xi\to 1_{[0,1]}$, $\delta\to 0$. 
\end{proof}

\section{Maximum principle for FBMHs in wedge domains}\label{Sec: elliptic PDEs}

In this appendix, we show the maximum principle for FBMHs in wedge domains, which is based on the maximum principle of Lieberman \cite{lieberman1987local}. 

\begin{theorem}\label{Thm: maximum principle}
    Let $M^{n+1}\subset \R^L$ be a locally wedge-shaped Riemannian manifold (with induced metric), $B\subset M$ be a simply connected open set. 
    Suppose $(\Sigma_i, \{\bd_m\Sigma_i\}) \subset (B, \{\bd_{m+1}M\cap B\})$, $i\in\{1,2\}$, are two connected $C^{1,1}$-to-edge almost properly embedded free boundary minimal hypersurfaces so that $\Sigma_1\cap\Sigma_2\neq\emptyset$ and $\Sigma_2$ lies on one side of $\Sigma_1$. 
    Then 
    \begin{itemize}
        \item either $\Sigma_1$ is transversal to $\Sigma_2$ at some touching point $p\in  (\interior(\Sigma_i)\cap\bd^FM)\cup \bd^{FE}\Sigma_i$, $i=1$ or $2$, 
        \item or $\Sigma_1=\Sigma_2$.
    \end{itemize}
\end{theorem}
\begin{proof}
    Suppose $\Sigma_1$ does not meet $\Sigma_2$ transversally at any touching point of $\Sigma_i$, $i\in\{1,2\}$. 
    Then for any $p\in \Sigma_1\cap\Sigma_2$ with $p\notin\bd^E\Sigma_1$, we must have $\Sigma_1$ and $\Sigma_2$ are touching at $p$ by Proposition \ref{Prop: classify FBMH}. 
    It follows from the standard maximum principle that $\Sigma_1\setminus\bd^E\Sigma_1=\Sigma_2\setminus\bd^E\Sigma_2$, and thus $\Sigma_1=\Sigma_2$ by taking closure. 
    
    Now we assume $p\in\bd^E\Sigma_1\cap\Sigma_2$. 
    Then the above assumptions and Proposition \ref{Prop: classify FBMH} immediately suggest that $\Sigma_1$ and $\Sigma_2$ are touching at $p\in\bd^E\Sigma_1\cap\bd^E\Sigma_2$. 
    Hence, in a local model $(\phi_{p}, \mB^L_{2R}(p), \Omega_{p}^n\times\R)$ of $M$ near $p$, we can write 
    \[ \Sigma_i = {\rm Graph}(u^{i}) = \{(x_1,\dots,x_n, u^{i}(x_1,\dots,x_n))\},\quad i\in\{1,2\},\]
    for some $u^{i}\in C^{1,1}(\Omega_R)\cap C^\infty(\Omega_R\setminus\bd^E\Omega)$ in a wedge domain $\Omega_R:= \Omega\cap \Clos(\mB^{n}_R(0))$ so that $u^i$ satisfies the minimal surface equation with oblique boundary condition (\ref{Eq: boundary condition of FBMH on horizontal plane}), and
    \begin{equation}\label{Eq: maximum principle 1}
        u^{1}(0)=u^{2}(0)=0, \quad Du^{1}(0)=Du^{2}(0)=0, \quad u^{1}(x)\geq u^{2}(x),~\forall x\in\Omega_R.
    \end{equation}
    By a standard argument (cf. \cite[Lemma 1.26]{colding2011course}), $v:=u^{1}-u^{2}\in C^{1,1}(\Omega_R)\cap C^{\infty}(\Omega_R\setminus\bd^E\Omega)$ satisfies a second order elliptic PDE:
    \[ \left\{ 
    \begin{array}{ll}
    \tilde{a}_{ij}v_{ij} + \tilde{b}_iv_i +\tilde{c}v = 0    & \mbox{in $\interior(\Omega_R)$}, \\
    \tilde{\beta}_{\pm}^iv_i + \tilde{\gamma}_{\pm} v = 0     & \mbox{on $\bd^{\pm}\Omega_R$},
    \end{array}\right. \]
    where $\tilde{a}_{ij}\in C^{0,1}(\Omega_R)$, $\tilde{b}_i,\tilde{c}\in L^\infty(\Omega_R)$, $\tilde{\beta}_{\pm}^i\in C^{1,1}(\bd^{\pm}\Omega_R)$, and $\tilde{\gamma}_{\pm}\in C^{0,1}(\bd^{\pm}\Omega_R)$. 
    Since $v\geq 0$ and $v(0)=0$, it follows from the maximum principle of Lieberman \cite[Corollary 2.4]{lieberman1987local} that $v\equiv 0$, i.e. $\Sigma_1=\Sigma_2$. 
\end{proof}

\bibliographystyle{abbrv}


\bibliography{reference.bib}   

%
%






\end{document}